\documentclass[10pt]{article}
 \usepackage[dvips]{graphicx}
 \usepackage{psfrag}
 \let\theoremstyle\relax
 \usepackage{amsfonts, amssymb, amsmath, amsthm}
 \usepackage{color}
\usepackage{graphicx, array, blindtext}
 \usepackage{mathabx}
\usepackage{euscript}
\usepackage{mathrsfs}
\usepackage{appendix}
\usepackage{enumitem}
\usepackage{lscape}
\usepackage{multirow}
\usepackage[table,xcdraw]{xcolor}
\usepackage{longtable}
\usepackage{graphicx, array, blindtext}
\usepackage[space]{grffile}
\usepackage{float}
\usepackage{caption}
\usepackage{subcaption}
\DeclareMathAlphabet{\mathpzc}{OT1}{pzc}{m}{it}

\renewcommand{\a}{\alpha}   \renewcommand{\b}{\beta}

   \newcommand{\e}{\epsilon}
\newcommand{\ve}{\varepsilon}
\newcommand{\g}{\gamma}      
   \renewcommand{\l}{\lambda}

\newcommand{\m}{\mu}

\renewcommand{\r}{\rho}      
   
\renewcommand{\th}{\theta}      
     
\newcommand{\z}{\zeta}   
\newcommand{\D}{\Delta} 

\newcommand{\diag}{\operatorname{diag}}

\newcommand{\osc}{\operatorname{osc.}}

\newcommand{\newvec}[1]{\vbox{\ialign{##\crcr$\displaystyle\rightharpoonup$\crcr\noalign{\kern-1pt\nointerlineskip}
$\hfil\displaystyle{#1}\hfil$\crcr}}} 
\newcommand{\longrightharpoonup}{\relbar\joinrel\rightharpoonup}
\newcommand{\longvec}[1]{\vbox{\ialign{##\crcr$\displaystyle\longrightharpoonup$\crcr\noalign{\kern-1pt\nointerlineskip}
$\hfil\displaystyle{#1}\hfil$\crcr}}} 

\theoremstyle{theorem}
\newtheorem{thrm}{Theorem}[section]
\newtheorem{dfn}[thrm]{Definition}

\newtheorem{remark}[thrm]{Remark}
\newtheorem{exam}[thrm]{Example}
\newcommand{\TV}{\text{T.V.}}

\newcounter{fig}
\newenvironment{fig}[1][]{\refstepcounter{fig}\par\medskip\noindent\center{}}

\title{Global Transonic Solutions to Combined Fanno Rayleigh Flows Through Variable Nozzles\footnote{Research is supported by the National
         Science Council of Taiwan and Center for Theoretical Sciences, Mathematical Division, NCU.}}

\author{Shih-Wei Chou\footnote{E-mail: swchou@math.ncu.edu.tw},\ John M. Hong\footnote{E-mail: jhong@math.ncu.edu.tw},\ Bo-Chih Huang\footnote{E-mail: huangbz@math.ncu.edu.tw},\
and Reyna Quita\footnote{E-mail: reynaquita2905@gmail.com.}
\\
         {\small Department of Mathematics, National Central University,}\\{\small Taoyuan 32001, Taiwan}
          }

\date{}

\begin{document}
\maketitle


\date{}


\centerline{\today}
\medskip

\begin{abstract}
In this paper, we study the initial-boundary value problem of compressible Euler equations with friction and heating that model the combined Fanno-Rayleigh flows through symmetric variable area nozzles, in particular, the case of contracting nozzles is considered. A new version of a generalized Glimm scheme (GGS) is presented for establishing the global existence of transonic entropy solutions. Modified Riemann and boundary Riemann solutions are applied to design this GGS which obtained by  the contraction matrices acting on the homogeneous Riemann (or boundary-Riemann) solutions. The extended Glimm-Goodman wave interaction estimates are investigated for ensuring the stability of the scheme and the positivity of gas velocity that leads to the existence of the weak solution. The limit of approximation solutions serves as an entropy solution. Moreover, a quantitative relation between the shape of the nozzle, the friction, and the heat is proposed. Under this relation, the global existence of weak solution for contracting nozzle is achieved. Simulations of contraction-expansion and expansion-contraction nozzles are presented to illustrate the relevant theoretical results.
\end{abstract}

\section{Introduction}
\setcounter{section}{1}
\label{sec1}
In 1982, Tai-Ping Liu \cite{Liu1} showed that, for the full compressible Euler equations of gas dynamics, flows along an expanding duct are always asymptotically stable, whereas flows with a standing shock wave along a contracting duct are dynamically unstable. Since then the global existence of time evolution solution for nozzle flow with contracting duct is unsolved for decades. In this paper, we prove that under certain appropriate friction and heat affect, the flow of entropy solution along the contracting duct exists globally. We consider the combined Fanno-Rayleigh flows through symmetric variable area nozzles, which are governed by the one-dimensional compressible Euler equations including those of friction and heating (see \cite{S,TS}):
\begin{eqnarray*}
&&(a \rho)_t + (a \rho u)_x =0,\\
&&(a \rho u)_t+(a \rho u^2+a P)_x =a_xP-\displaystyle \alpha \sqrt{a}\rho u|u|,\\
&&(a E)_t+(a u(E+P))_x =\beta a q(x)-\displaystyle \alpha \sqrt{a}\rho u^2|u|,
\end{eqnarray*}
where $\rho$, $u$, $E$ and $P$ are, respectively, the density, velocity, total energy and pressure of the gas, $\a$ is the coefficient of friction, $q(x)$ is a given function representing the heating effect from the force outside the nozzle. The cross section $a(x)$ of the nozzle is described by a Lipschitz continuous function. The model of Rayleigh flow ($\alpha=0$) has been studied frequently in the fields of fluid dynamics and aerodynamics, particularly regarding the design of aircraft engines. The model of Fanno flow ($\beta=0$) is always adopted in many engineering fields, such as stationary power plants, transport of natural gas in long pipe lines and for the design and analysis of fluid motion in micro-scaled nozzles ($0<a(x)\ll 1$) in which flow friction plays a crucial role. Classically, the Fanno flow model has been considered as the flow in constant duct with friction \cite{S,TS}. By the result of dimensional analysis, the friction term should be modified to the form of $\a\sqrt{a}\rho u|u|$ in variable area duct case. To compare the effect of friction with that of heating $q$, we impose constant $\beta$ in the presented system. This paper focuses on combined effects of the heat and friction that cause the change of flow motion along variable area ducts.\\

We express he relation between $P$ and $E$ as
$$
E=\frac{1}{2}\rho u^2+\frac{P}{\gamma -1},
$$
where $\gamma$ is the adiabatic constant satisfying $1<\gamma\leq\frac{5}{3}$. The presented system can then be written as the following hyperbolic system of balance laws for the mass, momentum and total energy:
\begin{subequations}
\label{1.1}
\begin{alignat}{2}
&\rho_t + (\rho u)_x =-\displaystyle\frac {a_x}{a} \rho u,\\
&(\rho u)_t+(\frac{3-\gamma}{2}\rho u^2+(\gamma-1)E)_x =-\displaystyle\frac {a_x}{a}
\rho u^2-\displaystyle\frac {\alpha}{\sqrt{a}}
\rho u|u|,\\
&E_t+(u(\gamma E-\frac{\gamma-1}{2}\rho u^2))_x
=-\displaystyle\frac {a_x}{a} u(\gamma E-\frac{\gamma-1}{2}\rho
u^2)-\frac{\alpha}{\sqrt{a}}\rho u^2|u| + \beta q(x).
\end{alignat}
\end{subequations}
Let
\begin{align}
\label{1.2.0}
 &m:=\rho u,\quad U:=(\rho , m, E), \quad h_1(x):=-\displaystyle\frac {a_x}{a}, \quad h_2(x):= - \frac{\alpha}{\sqrt{a}}, \\
 &F(U):=\left(m, \frac{3-\gamma}{2}\frac{m^2}{\rho}+(\gamma-1)E , \frac{m}{\rho}(\gamma E-\frac{\gamma-1}{2}\frac{m^2}{\rho})  \right),
 \quad \\
 &G(x,U):=\left(h_1 m,(h_1+h_2)
\frac{m^2}{\rho},h_1\frac{m}{\rho}(\gamma
E-\frac{\gamma-1}{2}\frac{m^2}{\rho})+h_2\frac{m^3}{\rho^2}+\beta q(x) \right).
\end{align}
System \eqref{1.1} is then written into the following compact form
\begin{equation}
\label{1.2.1}
U_t+F(U)_x=G(x,U).
\end{equation}
The initial-boundary value problems of \eqref{1.2.1}, subject to the initial and boundary data near the sonic states, is set forth as follows:
\begin{alignat}{3}
\label{1.2.2} \left\{ \setlength\arraycolsep{0.1em}
\begin{array}{ll}
U_t+F(U)_x=G(x,U),\quad x>x_B,\; t>0,\\
U(x,0)=U_0(x)\in \Omega,\\
U(x_B,t)=U_B(t),\; t>0,
\end{array}\right.
\end{alignat}
where
$U_0(x)=(\rho_0(x),m_0(x),E_0(x))=(\rho_0(x),\rho_0(x)u_0(x),E_0(x))$,
and
\begin{equation}
\label{1.2.3}
\Omega:=\{U|  \min\{m_B(t)\}\geq r^* \cdot T.V.\{U_0\}\quad \text{and} \quad\|U-U_0||_{L^{\infty}}\leq r\}
\end{equation}
is a ball of radius $r$ in $\mathbb{R}^2$, and $T.V.\{U\}$ is the total variation of $U$, centered at particular sonic state
$$
U_*\equiv(\rho_*,m_*,E_*)\in \mathcal{T}:=\{(\rho , m, E): \;m=\rho c,\;\rho \geq c>0 \},\ \text{where }c:=\sqrt{\gamma(\gamma-1)(\frac{E}{\rho}-\frac{u^2}{2})}.
$$
We call the curve $\mathcal{T}$ the transition curve. The boundary condition is given as
\begin{equation}
\label{bdcond}
U_B(t)=\left\{\begin{array}{lll}
(\r_B(t),m_B(t),E_B(t)), &\text{if } u(x_B,0)>c(x_B,0), & \text{(supersonic boundary)} \\
(\r_B(t),m_B(t)),        &\text{if } u(x_B,0)\le c(x_B,0), & \text{(subsonic boundary)}.
\end{array}\right.
\end{equation}
\\

In this paper, we propose a new generalized Glimm method for establishing the global existence of transonic entropy solutions for \eqref{1.2.2}. Throughout this paper, we impose the following conditions:
\begin{enumerate}
\item [(A$_1$)] $\rho_0(x)$, $m_0(x)$ (or $u_0(x)$) and $E_0(x)$ are bounded positive functions with small total variations, and $q(x)\geq 0$ for all $x\geq x_B$;
\item [(A$_2$)] there exists $a^*>0$ such that for every $t\geq 0$, we have
$$
\|h_{1}'\|_{L^1}\leq a^*, \quad \|h_{2}'\|_{L^1}\leq a^* \quad\text{and}\quad \|q' \|_{L^1}\leq a^*,
$$
where $\|h'\|_{L^1}\equiv \int_{x_B}^{\infty}h'(s)ds$;
\item [(A$_3$)] $\rho_0(x)$, $u_0(x)$, $h_1$, $h_2$ and $q$ satisfy
$$
\frac{7-\gamma}{3}{h_1} u_0+ \frac{4}{3}{h_2} u_0 -\frac{\gamma(\gamma-1)}{\rho_0 c^2_0}\beta q<0
$$
for $x_B\leq x<\infty$.
\item [(A$_4$)] for any $0<\epsilon\ll1$, there exists a positive constant $\mathcal{C}$ depending on the initial and boundary data such that
$$
\min\limits_{t\geqslant 0}\{m_B(t)\}>(1+\epsilon)T.V.\{U_0\}+(1+\epsilon+\epsilon^2)^2\mathcal{C},
$$
where $T.V.\{U\}$ is the total variation of $U$ on $x_B\leq x<\infty$.
\end{enumerate}

We review previous results related to this topic and clarify the motivation of the study. When the duct is uniform and the effects of friction and heat are neglected, system \eqref{1.2.1} is reduced to a strictly hyperbolic system of conservation laws
\begin{equation}
\label{1.4}
U_t + F(U)_x=0.
\end{equation}
The entropy solutions to the Riemann problem of \eqref{1.4} were first constructed by Lax \cite{LA1}. The solutions are self-similar and consist of constant states separated by elementary waves: rarefaction waves, shocks and contact discontinuities. Furthermore, the global existence of weak solutions to the Cauchy problem of \eqref{1.4} was established by Glimm \cite{GL}, who applied Lax's solutions as the building blocks of a finite difference scheme. For the quasi-linear hyperbolic system of balance laws
\begin{eqnarray}
\label{1.5.1}
 U_t +F(x,U)_x = G(x,U),
\end{eqnarray}
the Cauchy problem was first studied by Liu \cite{TP1}, who showed that if the eigenvalues of $\partial_U F$ are nonzero and the $L^1$-norms of $G$ and $\partial_U G$ are sufficiently small, then weak solutions exist globally and tend asymptotically to stationary solutions. If $F$ and $G$ are independent of $x$, the existence result for the nonlinear waterhammer problem was established by Luskin and Temple \cite{LT} by combining Glimm's scheme with the method of fractional steps. Recently, the authors in \cite{HCHY} showed the global existence of the transonic entropy solution of \eqref{1.2.1} where $G$ is composed of the gravity and heat. For the Cauchy problem of the general quasi-linear, strictly hyperbolic system of balance laws is expressed as
\begin{eqnarray}
\label{1.5.3}
 U_t +F(x,t,U)_x = G(x,t,U).
\end{eqnarray}
The local existence of entropy solutions was first established by Dafermos-Hsiao \cite{DH} under the assumption that the eigenvalues of $\partial_U F$ are nonzero and the constant solution $U\equiv0$ is the steady state solution for all $(x,t)$. Furthermore, the global existence was also obtained under additional dissipative assumptions regarding the flux and source. In \cite{CHS}, the dissipative assumption in \cite{DH} was relaxed to obtain global existence results for the Cauchy problem of nozzle flow. System \eqref{1.5.3} has also been studied by LeFloch-Raviart
\cite{LR} and Hong-LeFloch \cite{HL} using an asymptotic expansion around the classical Riemann solutions. The shock wave model of Einstein's equations, which can be written as a degenerate $4 \times 4$ hyperbolic balance law \eqref{1.5.3}, was studied by Groah-Smoller-Temple \cite{GST} by applying fractional time-step scheme.

We notice that, system \eqref{1.1} can be considered as a $4\times 4$ resonant system
\begin{align}
\label{1.5.2}
\begin{split}
 a_t&=0, \\
 U_t +F(a(x),U)_x& = a'(x)G(a(x),U)
\end{split}
\end{align}
whose eigenvalues of the Jacobian matrix for flux coincide in the sonic states at which the strict-hyperbolicity of \eqref{1.5.2} fails. The Cauchy problem for the $2 \times 2$ resonant system has been studied by Isaacson-Temple \cite{IT1, IT2} and Hong-Temple \cite{HT,HT1}, whereas generalized Glimm method for $n \times n$ strictly hyperbolic case have been proposed in \cite{JH, LF1, LL}. In particular, in \cite{JH} the author showed that the residual case only converges weakly in $L_1$. Research on resonant systems goes back to Liu \cite{Liu2}, Temple \cite{TE1}, and Goatin-LeFloch \cite{GOL}. Additional results on nonconservative systems, have been published in \cite{LF, LL, DA}.\\

The presented results have substantially contributed to research on both $n \times n$ strictly hyperbolic systems and $2\times 2$ resonant systems. However, no extended Glimm method thus far has yielded a satisfactory global existence result  for to the compressible Euler equations of transonic Fanno-Rayleigh flow, which can be reformulated as a $4\times 4 $ resonant system. Because of possible resonance, spatial BV-norm estimates are not generally available. Instead, the compensated compactness method is used only for studying $L^{\infty}$ solutions to the transonic flow, which can be reviewed in \cite{CSW,MO}. In this paper, we prove that there exist BV entropy solutions of the transonic Fanno-Rayleigh flow through new generalized Glimm method.

In contrast to previous methods, this new method considers the appearance of the source terms in \eqref{1.2.1} as the cause of perturbations to the self-similar Riemann or boundary Riemann solutions of homogeneous conservation laws. Therefore, through the techniques of linearization, the approximate solution $U$ of the Riemann or boundary Riemann problems, which are the building blocks of the generalized Glimm scheme, is constructed by combining the traditional Riemann (or boundary Riemann) solution $\widetilde{U}$ with the perturbation $\overline{U}$ to solve the linearized system around $\widetilde{U}$. Under this construction, $U$ is no longer self-similar. By using the operator-splitting method for the linearized problem of $\overline{U}$ and by using the averaging process to smooth the shocks in $\widetilde{U}$, we can reformulate $U$ as
$$
U(x,t)=S(x,\Delta t,\widetilde{U}(x,t))\cdot \widetilde{U}(x,t)
$$
for a $3 \times 3$ matrix $S(x,\Delta t,\widetilde{U})$ called the contraction matrix. Thus, we present a new type of approximate solution which has the structure of elementary waves connected by functions as background states. Moreover, using this as the building block of a GGS, we can construct the approximate solutions $\{U_{\Delta x}\}$ of \eqref{1.2.2}. For the stability of scheme, the uniform bound of the total variation of
$\{U_{\Delta x}\}$ is obtained through two steps:
(1) proving that Glimm functional is non-increasing in time, and
(2) showing that all elements of the set $\{\overline{U}_{\Delta x}\}$ have finite total variation in each time step. The non-increasing of Glimm functional can be calculated according to the modified wave interaction estimates. Contrary to the method used in \cite{CHS, CHS1, CHS2, HCHY}, in which used 1-norm to estimate the wave strength, we estimate the wave strength in vector form by 2-norm to obtain more accurate to the arc length of the wave curve (see Theorem \ref{thm3.1}). Because of the effect of $S(x,\Delta t,\widetilde{U}_{\Delta x}(x,t))$, such wave interaction estimates yield a new dissipative relation between the wave strengths of the locally outgoing and incoming waves in $\{\widetilde{U}_{\Delta x}\}$. Furthermore, we show that the velocity $\{u_{\Delta x}=\frac{m_{\Delta x}}{\rho_{\Delta x}}\}$ is positive for all values of time, thus leading to a weaker dissipative condition (see ($A_3$)) than that in \cite{DH} for the stability of the scheme.

The new dissipative condition ($A_3$) gives us a quantitative relation between the shape of the nozzle which tells us that with certain specific friction, depending  on the shape of the duct, the time evolution solution for \eqref{1.2.2} exists globally even in the contracting duct case. Therefore, we conjecture that the traditional 1-dimensional inviscid nozzle flow equation is a critical model case for modeling the fluid in the nozzle. Moreover, compared to \cite{LXY} in which the shape of the nozzle is $O(x^3)$, condition ($A_3$) implies that with friction effect, the shape of the nozzle could be $O(x^2)$ in 1-dimensional case. Indeed, the numerical simulation in Example \ref{exam5.2} indicates that the initial shock front can pass through the contraction-expansion nozzle when the duct has friction. Example \ref{exam5.3} demonstrate that different shape of nozzles allow different friction parameter $\a$ to obtain the existence of flows. Throughout this paper, we rule out the vacuum case to avoid resonance.\\

Here, we provide the definitions of the weak and entropy solutions of problem \eqref{1.2.2}, and state the main theorem.
\begin{dfn}
\label{def1.1}
Consider the initial-boundary value problem \eqref{1.2.2} satisfying conditions $(A_1)$-$(A_4)$. We say that a bounded measurable function $U$ is a weak solution of \eqref{1.2.2} if
\begin{equation}
\label{1.7} \iint_{x>x_B,t>0}
\left\{U\phi_t+F(U)\phi_x+{G}(x,U)\phi\right\} dx dt
+\int^{\infty}_{x_B} U_0(x)\phi(x,0) dx+\int^{\infty}_{0}
f(U_B(t))\phi(x_B,t) dt=0
\end{equation}
for any test function ${\phi} \in C^1_0([x_B,\infty)\times[0,\infty))$.
\end{dfn}

\begin{dfn}
\label{def1.2}
Let $\Omega$ be a convex subset of $\mathbb{R}^2$. We say that $(\eta(U),q(U))$ is an entropy pair of \eqref{1.1} provided that $\eta$ is convex on $\Omega$ and
\begin{eqnarray}
\label{1.8}
{Dq}={D\eta}\cdot{DF}\quad\text{on } \Omega.
\end{eqnarray}
Furthermore, a bounded measurable function $U$ is called an entropy solution of \eqref{1.2.2} if $U$ is a weak solution of \eqref{1.2.2} and satisfies
\begin{equation}
\label{1.9}
\displaystyle\iint_{x>x_B,t>0}\left\{\eta \phi_t+q\phi_x+{D\eta}\cdot{G}\phi\right\} dxdt
\displaystyle+\int^\infty_{x_B}\eta(U_0(x))\phi(x,0)dx+\int^{\infty}_{0} q(U_B(t))\phi(x_B,t) dt \geq 0
\end{equation}
for every entropy pair $(\eta(U),q(U))$ and positive test function ${\phi} \in C^1_0([x_B,\infty)\times[0,\infty))$.
\end{dfn}\;

\noindent {\bf Main Theorem.}
{\it Consider the initial-boundary value problem \eqref{1.2.2} satisfying conditions {\rm(}A$_1${\rm)}-{\rm(}A$_4${\rm)}. Let $\{U_{\theta,\Delta x}\}$ be the sequences of approximate solutions of \eqref{1.2.2} through the generalized Glimm scheme, there exist a null set $N \subset \Phi$ and a subsequence $\{\Delta x_{i}\}\rightarrow 0$ such that if $ \theta \in \Phi\setminus N$, then $$U(x,t):= \lim_{\scriptstyle \Delta x_i\rightarrow 0} U_{\theta,\Delta x_i}(x,t)$$ is an positive entropy solution of \eqref{1.2.2}. In particular, the velocity $u(x,t)>0$ for all $(x,t)\in [x_B,\infty)\times[0,\infty)$.}\\

The remainder of is paper is organized as follows. In Section \ref{sec2}, we construct the approximate solutions of the generalized Riemann and boundary-Riemann problems, which are combinations of the classical Riemann solutions and perturbation terms for solving the linearized system around the Riemann and boundary-Riemann solutions of homogeneous conservation laws. In addition, for the consistency of the GGS, we estimate the residuals of such approximate solutions. In Section \ref{sec3}, we present the GGS whose building blocks are the approximate solutions constructed in Section \ref{sec2}. We prove the modified version of Glimm-Goodman's wave interaction estimates. In Section \ref{sec4}, we prove the non-increasing of Glimm functionals in time and the uniform boundedness of total variations for the perturbation terms in each time step for the stability of GGS. Finally, we obtain the consistency of GGS by showing that the residual approaches zero as grid sizes approach zero. In addition, the entropy inequality is shown. So, we establish the global existence of the entropy solution to \eqref{1.2.2}. The numerical simulations are presented in Section \ref{sec5}.\

\section{Generalized solutions of Riemann and boundary-Riemann problems}
\setcounter{equation}{0}
\label{sec2}

In this section, we construct the approximate solutions to the Riemann and boundary-Riemann problems of \eqref{1.1}. First, for given $\Delta x$, $\Delta t>0$ satisfying the Courant-Friedrichs-Levy condition
\begin{equation}
\label{2.0.0.0}
\displaystyle\frac{\Delta x}{\Delta t} > \sup\limits_{(\rho,m)
\in\Omega}\left\{\frac{m}{\rho}+\sqrt{\gamma(\gamma-1)(\frac{E}{\rho}-\frac{u^2}{2})}\right\},
\end{equation}
we define the regions
\begin{align}
\label{ingrid}
D(x_0,t_0)&:=\{(x,t)\mid|x-x_0|<\Delta x, \;t_0<t< t_0+\Delta t\},
\end{align}
and
\begin{align}
\label{bgrid}
D_B(x_B,t_0)&:=\{(x,t)\mid x_B<x<x_B+\Delta x, \;  t_0<t< t_0+\Delta t\}.
\end{align}
Then, the Riemann problem of \eqref{1.1}, denoted by $\mathcal{R}_G(x_0,t_0)$, is given by
\begin{equation}
\label{2.0}\left\{
\begin{array}{ll}
U_t+F(U)_x=G(x,U),\quad (x,t)\in D(x_0,t_0),\medskip\\
U(x,t_0) = \left\{
\begin{array}{lc}
 U_L,&  \text{if }\; x_0-\Delta x<x<x_0,\\
 U_R,&  \text{if }\; x_0<x<x_0+\Delta x.
 \end{array}\right.
\end{array}\right.
\end{equation}
The boundary-Riemann problem of \eqref{1.1}, denoted by
$\mathcal{BR}_G(x_B,t_0)$, is given by
\begin{equation}
\label{BRP}
\left\{\begin{array}{ll}
U_t+f(U)_x=G(x,U),         &(x,t)\in D_B(x_B,t_0), \\
U(x,t_0)=U_R,                  &x_B\le x\le x_B+\D x, \\
U(x_B,t)=U_B, &t_0<t<t_0+\D t,
\end{array}\right.
\end{equation}
where $m$, $U$, $F$ and $G$ are given in \eqref{1.2.0}, $U_L=(\rho_L, m_L, E_L)$, $U_R=(\rho_R, m_R, E_R)$ are positive constant states, and $U_B$ are positive constant state $(\r_B,m_B)$ or $(\r_B,m_B,E_B)$ as in \eqref{bdcond}. By setting source term $G \equiv 0$ in \eqref{2.0}, we have the corresponding classical Riemann and boundary-Riemann problems denoted by $\mathcal{R}_C(x_0,t_0)$ and $\mathcal{BR}_C(x_B,t_0)$ respectively.

\subsection{Approximate generalized solutions of $\mathcal{R}_G(x_0,t_0)$ and $\mathcal{BR}_G(x_B,t_0)$}
The purpose of this subsection is to provide the approximate generalized solutions of $\mathcal{R}_G(x_0,t_0)$ and $\mathcal{BR}_G(x_B,t_0)$. We first start at the entropy solution of $\mathcal{R}_C(x_0,t_0)$ and $\mathcal{BR}_C(x_B,t_0)$. System \eqref{1.2.1} is a strictly hyperbolic system whose Jacobian matrix $DF$ has three distinct real eigenvalues
$$
\lambda_1(U):= u-c(U),\quad \lambda_2(U):=u,\quad \lambda_3(U):=u+c(U),
$$
where
$$
u=\frac{m}{\rho},\quad c(U)=\sqrt{\gamma(\gamma-1)(\frac{E}{\rho}-\frac{u^2}{2})}.
$$
Here $c(U)$ is called the sound speed. We define the curve of sonic states as $$\mathcal{T}:=\{U;\; u=c(U)\}.$$ The corresponding eigenvectors of $\lambda_1$, $\lambda_2$ and $\lambda_3$ are, respectively,
\begin{eqnarray*}
R_1(U) &=& (-1,c-u,uc-H)^T, \\
R_2(U) &=& (1,u,\frac{u^2}{2})^T, \\
R_3(U) &=& (1,c+u,uc+H)^T,
\end{eqnarray*}
where
$$
H=H(U)=\frac{\gamma E}{\rho}-\frac{\gamma-1}{2}u^2
$$
is the total specific enthalpy. It is easy to calculate that
\begin{equation*}
\nabla\lambda_i(U)\cdot R_i(U)=\frac{(\gamma+1)c}{2\rho}>0,\
i=1,3,\quad\text{and}\quad\nabla\lambda_2(U)\cdot R_2(U)=0,
\end{equation*}
which implies the 1st and 3rd characteristic fields are genuinely nonlinear and the 2nd characteristic field is linear degenerate. Therefore, the entropy solutions of $\CMcal{R}_C(x_0,t_0)$ and $\CMcal{BR}_C(x_B,t_0)$ consist of either shock waves, rarefaction waves from genuinely nonlinear fields or contact discontinuities from the linear degenerate field. For $i \in \{1,3\}$, each $i$-rarefaction wave is a self-similar function
$$
U=U(\zeta), \quad \zeta=\frac{x-x_0}{t-t_0},
$$
also it satisfies
\begin{equation}
\label{2.0.2}
(DF(U)-\zeta I_3)\cdot \frac{dU}{d\zeta}=0,
\end{equation}
where $I_3$ is the $3\times3$ identity matrix, the $i$-shock is a discontinuous solution satisfying the Rankine-Hugoniot condition
\begin{equation}
\label{2.0.3}
s [U]=[F(U)]
\end{equation}
and Lax's entropy condition
\begin{equation}
\label{2.0.4}
\lambda_i(U_R) <s< \lambda_i(U_L),
\end{equation}
where $s$ is the speed of the $i$-shock and $[\cdot]$ denotes the difference of states across the shock. For the later use, we let $\mathcal{R}_i(U_L)$ and $\mathcal{S}_i(U_L)$ denote, respectively, the $i$-shock and the $i$-rarefaction wave curves starting at $U_L$. For $i=2$, the solution consists of the $i$-contact discontinuity satisfying \eqref{2.0.3} and it behaves like the linear discontinuous waves due to the original
discontinuity of the Riemann data. By Lax's method in \cite{LA1,SM}, we can show the existence and uniqueness of the entropy solution to $\mathcal{R}_C(x_0,t_0)$. The solution consists of at most 4 constant states separated by either shock waves, rarefaction waves or a contact discontinuity.\

For the classical boundary-Riemann problem $\mathcal{BR}_C(x_B,t_0)$, when $U_B$, $U_R$ are near sonic states $\mathcal{T}$, the entropy solutions may not be unique even conditions \eqref{2.0.3}, \eqref{2.0.4} are imposed. In addition, some solutions may have large total variations even when $|U_R-U_B|$ is small. To overcome the difficulty, it is necessary to impose the following extra condition for the admissible weak solutions of $\mathcal{BR}_C(x_B,t_0)$:
\begin{enumerate}
\item[($\mathcal{E}$)] Weak solution $U=(\rho, m,E)$ is the entropy solution of $\mathcal{BR}_C(x_B,t_0)$ if $U$ has the least total variation in $\rho$ within all weak solutions of $\mathcal{BR}_C(x_B,t_0)$.
\end{enumerate}
Under condition $(\mathcal{E})$, we are able to select the unique entropy solution of $\mathcal{BR}_C$, moreover, the entropy solution does not consist of zero speed 1-shock attaching on $x=x_B$. We have the following theorem for the existence and uniqueness of entropy solutions of $\mathcal{R}_C(x_0,t_0)$ and $\mathcal{BR}_C(x_B,t_0)$.

\begin{thrm}
\label{thm2.1}
{\rm(\cite{LA1, SM})} Consider problems $\mathcal{R}_C(x_0,t_0)$ and $\mathcal{BR}_C(x_B,t_0)$. Let $U_L \in\Omega$. Then there is a neighborhood $\widetilde{\Omega}\subset\Omega$ of $U_L$ such that if $U_R\in\widetilde{\Omega}$, then $\mathcal{R}_C(x_0,t_0)$ has a unique solution consisting of at most four constant states separated by shocks or rarefaction waves and contact discontinuity. Moreover, under condition $(\mathcal{E})$, there exists $E_B>0$ and $U_B=(\rho_B, m_B,E_B)\in\widetilde{\Omega}$ such that $\mathcal{BR}_C(x_B,t_0)$ admits a unique solution $U$ satisfying $U(x_B,t)=U_B$.
\end{thrm}

Next, based on Theorem \ref{thm2.1}, we are ready to construct the approximate solutions of $\mathcal{R}_G(x_0,t_0)$ and $\mathcal{BR}_G(x_B,t_0)$. We begin with $\mathcal{R}_G(x_0,t_0)$. Let $\widetilde{U}=(\widetilde{\rho},\widetilde{m},\widetilde{E})^T=(\widetilde{\rho},\widetilde{\rho}\widetilde{u},\widetilde{E})^T$ be the entropy solution of $\mathcal{R}_C(x_0,t_0)$. Then we construct the approximate solution $U$ of $\mathcal{R}_G(x_0,t_0)$ by
\begin{align}
\label{2.7.1}
U(x,t)&=\widetilde{U}(x,t) + \overline{U}(x,t)
\end{align}
for $(x,t)\in D(x_0,t_0)$ where $\overline{U}(x,t)$ is the perturbation term needed to be decided due to the appearance of the source terms. The detail of the construction for $\overline{U}$ will be given as follows. First, we consider the linearized system of \eqref{1.2.1} around $\widetilde{U}$ with initial data $\overline{U}(x,0)= 0$:
\begin{align} 
\label{2.8} \left\{
   \begin{array}{ll}
    \overline{U}_t+(A(x,t)\overline{U})_x=B(x,t)\overline{U}+
C(x,t), & (x, t)\in D(x_0,t_0),\\
     \overline{U}(x,0)= 0, & x_0-\Delta x\leq x\leq x_0+\Delta x,
   \end{array}
   \right.
\end{align}
where
{\color{black}
\begin{eqnarray*}
&& A(x,t)=dF(\widetilde{U})=\left[\begin{array}{ccc}
0 & 1 & 0 \\
\frac{\gamma-3}{2}\tilde{u}^2 & (3-\gamma)\tilde{u} & \gamma-1 \\
\vspace{-0.4cm} \\
\frac{\gamma-1}{2}\tilde{u}^3-\tilde{u}\widetilde{H} & \widetilde{H}-(\gamma-1)\tilde{u}^2 & \gamma\tilde{u}
\end{array}\right], \\
&& B(x,t)=G_U(x,\widetilde{U})=\left[\begin{array}{ccc}
0 & h_1 & 0 \\
-(h_1+h_2)\tilde{u}^2 & 2(h_1+h_2)\tilde{u} & 0 \\
h_1 ( \frac{\gamma - 1}{2}\tilde{u}^3 - \tilde{u} \widetilde{H}) -2 h_2 \tilde{u}^3 & h_1(\widetilde{H}-(\gamma-1)\tilde{u}^2 ) + 3 h_2 \tilde{u}^2 & h_1 \gamma \tilde{u}
\end{array}\right], \\
&& C(x,t)=G(x,\widetilde{U})=\left[\begin{array}{c}
h_1 \tilde{m} \\ (h_1+h_2)\tilde{\rho} \tilde{u}^2 \\ h_1 \tilde{m}\widetilde{H}+h_2 \tilde{m} \tilde{u}^2 +\beta q
\end{array}\right].
\end{eqnarray*}
}
Applying the operator-splitting method to \eqref{2.8} and observing that the homogeneous problem \eqref{2.8} admits the zero solution, we can approximate the solution of \eqref{2.8} by the one solving
\begin{align} 
\label{2.8.10}
   \left\{\begin{array}{ll}
    \overline{U}_t=B(x,t) \overline{U}+C(x,t), & (x, t)\in D(x_0,t_0),\\
     \overline{U}(x,0)= 0, &  x_0-\Delta x\leq x\leq x_0+\Delta
     x.
   \end{array}\right.
\end{align}
To obtain the better regularity of approximate solutions for \eqref{2.8.10}, the averaging process of the coefficients in \eqref{2.8.10} with respect to $t$ over $[0,\D t]$ is used. For a bounded variation function $w(x,t,U)$, we define the average of $w(x,t,U)$ as
\begin{equation}
\label{2.8.11} w_*(x):=\frac{1}{\Delta t}\int_{t_0}^{t_0+\Delta
t}w(x,s,U(x,s))ds, \quad x_0-\Delta x\leq x\leq x_0+\Delta x.
\end{equation}
Note that $w_*(x)$ is continuous even across the shock. Then, \eqref{2.8.10} is modified into the following problem:
\begin{align} 
\label{2.8.12} \left\{
   \begin{array}{ll}
    \overline{U}_t=B_*(x) \overline{U}+C_*(x), & (x, t)\in D(x_0,t_0),\\
     \overline{U}(x,0)= 0, &  x_0-\Delta x\leq x\leq x_0+\Delta   x,
   \end{array}
   \right.
\end{align}
where $B_*(x)$ and $C_*(x)$ are obtained by the average process \eqref{2.8.11}. To solve \eqref{2.8.12}, the eigenvalues of matrix
$B_*(x)$ are
\begin{equation}
\label{2.8.3} \sigma_1:= \gamma h_1 \tilde{u}_*,\quad \sigma_2 :=
(h_1 + h_2 - \sqrt{h_1 h_2 + h_2^2})\tilde{u}_*,\quad \sigma_3 :=
(h_1 + h_2 + \sqrt{h_1 h_2 + h_2^2})\tilde{u}_*,
\end{equation}
and the corresponding generalized eigenvectors are $P_1(x)=(0,0,1)^T$, and
\begin{eqnarray*}
&&P_2(x)=\Big(h_1(\sigma_2-\sigma_1),\ \sigma_2(\sigma_2-\sigma_1), \\
&&          \hspace{1.7cm}-h_1(\frac{\sigma_1}{\gamma}(H-\frac{(\gamma -1) \tilde{u}^2}{2}) + \sigma_2((\gamma - 1)\tilde{u}^2-H)
            +h_2(3 \sigma_2 - 2 h_1 u)u^2\Big)^T, \\
&&P_3(x)=\Big(h_1(\sigma_3-\sigma_1),\ \sigma_3(\sigma_3-\sigma_1), \\
&&          \hspace{1.7cm}-h_1(\frac{\sigma_1}{\gamma}(H-\frac{(\gamma -1) \tilde{u}^2}{2}) + \sigma_3((\gamma - 1)\tilde{u}^2-H)
            +h_2(3 \sigma_3 - 2 h_1 u)u^2\Big)^T.
\end{eqnarray*}
Therefore, the transformation matrix of $B_*(x)$ is $P(x)=[P_1(x),P_2(x),P_3(x)]$, and the Jordan form of $B_*(x)$ is of the form
$$
J(x)=P^{-1}(x)B_*(x)P(x)=\diag[\sigma_1,\sigma_2,\sigma_3].
$$
It means that, a fundamental matrix $e^{J(x)t}$ solving $\dot{X}(t)=J(x)X(t)$ is
\begin{align}
\label{2.11}
e^{J(x)t}=\diag[e^{\sigma_1t},e^{\sigma_2t},e^{\sigma_3t}].
\end{align}
By the transformation matrix $P(x)$, we obtain the state transition matrix of $\dot{X}(t)=B(x)X(t)$, which is given as
$$
N(x,t,s)=P(x)e^{J(x)(t-s)}P(x)^{-1}.
$$
Following the variation of constant formula, we obtain the solution of \eqref{2.8.10} given by
\begin{eqnarray}
\label{avgperturb}
&& \overline{U}_*(x,t)=\int_{t_0}^{t_0+t}N(x,t,s)C_*(x)ds\\
&& \hspace{1.5cm}=\left[\begin{array}{c}
(\frac{e^{\sigma_2 t} + e^{\sigma_3 t}}{2} +\frac{e^{\sigma_2 t} - e^{\sigma_3 t}}{2}\frac{\sqrt{h_2(h_1+h_2)}}{h_1+h_2} - 1)\tilde{\rho}_*\\
(\frac{e^{\sigma_2 t} + e^{\sigma_3 t}}{2}-1)\tilde{m}_* \\
n_1\tilde{m}_*+(n_2-1) \tilde{E}_*+\frac{(e^{\sigma_1 t}-1)\beta
q(x)}{\sigma_1}
\end{array}\right], \nonumber
\end{eqnarray}
where
\begin{eqnarray*}
n_1(x,t)&=&\frac{e^{\sigma_1 t}-1}{2}(1+\frac{\gamma h_2-2 h_1(\gamma -1)^2}{(\gamma -1)^2(h_1+h_2)-h_2 \gamma^2})\tilde{u}+\frac{e^{\sigma_2 t}-1}{4}\frac{\sigma_1(3\gamma \sigma_2-2\sigma_1)(\gamma-1)}{\gamma^2 \sigma_2(\sigma_1-\sigma_2)}\tilde{u}\\
&&+\frac{e^{\sigma_3 t}-1}{4}\frac{\sigma_1(3\gamma \sigma_3-2\sigma_1)(\gamma-1)}{\gamma^2 \sigma_3(\sigma_1-\sigma_3)}\tilde{u}\\
n_2(x,t)&=&e^{\sigma_1 t}+\frac{(e^{\sigma_1 t}-1)\gamma h_2}{(\gamma-1)^2(h_1+h_2)-\gamma^2 h_2}+\frac{e^{\sigma_2 t}-1}{2}\frac{\sigma_1(\sigma_1-\gamma\sigma_2)}{\gamma\sigma_2(\sigma_1-\sigma_2)}+\frac{e^{\sigma_3 t}-1}{2}\frac{\sigma_1(\sigma_1-\gamma\sigma_3)}{\gamma\sigma_3(\sigma_1-\sigma_3)}
\end{eqnarray*}
for $t\in[0,\Delta t]$. Replacing $\widetilde{U}_*$, $\widebar{U}_*$ in \eqref{avgperturb} by $\widetilde{U}$ and $\widebar{U}$, we finally obtain
\begin{equation}
\label{perturbsol}
\overline{U}(x,t)=(S(x,t,\widetilde{U})-I_3)\widetilde{U},
\end{equation}
On the basis of these steps and \eqref{2.7.1} the approximate solution $U(x,t)$ can be expressed as
\begin{equation}
\label{aproxsol4}
U(x,t)=S(x,t,\widetilde{U})\widetilde{U},
\end{equation}
where
\begin{equation}
\label{solver1}
S(x,t,\widetilde{U})=\left[\begin{array}{ccc}
\frac{e^{\sigma_2 t} + e^{\sigma_3 t}}{2} +\frac{e^{\sigma_2 t} - e^{\sigma_3 t}}{2}\frac{\sqrt{h_2(h_1+h_2)}}{h_1+h_2} & 0 & 0 \\
\vspace{-0.4cm} \\
0 & \frac{e^{\sigma_2 t} + e^{\sigma_3 t}}{2} & 0 \\
\frac{(e^{\sigma_1 t}-1)\beta q(x)}{\tilde{\rho}\sigma_1} & n_1 & n_2
\end{array}\right].
\end{equation}

The above averaging process is reasonable due to the facts that
\begin{equation*}
\iint_{D(x_0,t_0)}|\overline{U}-\overline{U}_*|dxdt=O(1)\left((\Delta
t)^3+(\Delta t)^2\cdot T.V._{D(x_0,t_0)}\{\widetilde{U}\}\right),
\end{equation*}
and
\begin{align}
\label{1stubar}
&\overline{U}=(S-I_3)\widetilde{U}\nonumber\\
&=\Big(h_1 \tilde{m}, (h_1+h_2)\tilde{u} \tilde{m},h_1 \tilde{m}(\frac{\tilde{c}^2}{\gamma - 1} + \frac{\tilde{u}^2}{2} )+h_2 \tilde{u}^2 \tilde{m}+ \beta q(x)\Big)^T\Delta t+O(1)(\Delta t)^2,
\end{align}
which leads to
\begin{equation}
\label{2.13.1}
|\overline{U}|=O(1)(\Delta t),
\end{equation}
and $\overline{U}\rightarrow 0$ as $t\rightarrow t_0$ and $h_1(x),\; h_2(x)\rightarrow 0$. This is consistent with the case of homogeneous hyperbolic conservation laws. Moreover, \eqref{2.7.1} still holds when $\widetilde{U}$ is a constant solution.

The construction for the approximate solution of $\CMcal{BR}_G(x_B,t_0;g)$ is similar to that for $\CMcal{R}_G(x_0,t_0;g)$. It means that the approximate solution of $\CMcal{BR}_G(x_B,t_0;g)$ is given by \eqref{aproxsol4} where $\widetilde{U}$ is the solution of $\CMcal{BR}_C(x_B,t_0)$. Note
that $\widebar{U}$ in $\CMcal{BR}_G(x_B,t_0;g)$ may not satisfy $\widebar{U}(x_B,t)=0$ because of \eqref{1stubar}, which means that approximate solution $U(x,t)$ on $x=x_B$ may not match boundary condition $U_B(t)$. However, by \eqref{2.13.1}, the difference between approximate solution
$U(x_B,t)$ and the boundary data $U_B$, $U_B(t)$ can be estimated by
\begin{eqnarray}
\label{2.14.0}
|U(x_B,t)-U_B| \le O(1)(\Delta t),
\end{eqnarray}
\begin{eqnarray}
\label{2.14}
&&|U(x_B,t),m(x_B,t)-U_B(t)|\le|\widebar{U}(x_B,t)|+|\widetilde{U}(x_B,t)-U_B(t)|\nonumber\\
&&\hspace{4.34cm}\le O(1)(\Delta t)+\underset{[t_0, t_0+\Delta t]}{\osc}\{{U}(x_B,\cdot)\},
\end{eqnarray}
where $\underset{I}{\osc}\{{U}(x_B,\cdot)\}$ denotes the oscillation of ${U}(x_B,\cdot)$ on $I$. We will show in the later sections that such construction for the solution of $\CMcal{BR}_G(x_B,t_0;g)$ does preserve the stability and consistency of the generalized Glimm scheme.

\subsection{Residuals of approximate solutions for Riemann and boundary-Riemann problems}

In this subsection, to obtain the consistency of the generalized Glimm scheme, we calculate the residuals for the approximate solutions of the Riemann and boundary-Riemann problems. Given a measurable function $U$, closed region $\Gamma \subset[x_B,\infty)\times[0, \infty)$ and test function $\phi\in C^1_0(\widebar{\Gamma})$ where $\Gamma \subset \widebar{\Gamma}$, we define the residual of $U$ for \eqref{1.2.1} in $\Gamma$ by
\begin{equation}
\label{res1}
R(U, \Gamma,\phi):=\iint_{\Gamma}\left\{U\phi_t+f(U)\phi_x+G(x,U)\phi\right\}dxdt.
\end{equation}
\begin{thrm}
\label{thm2.2}
Let $U$ and $U^B$ denote the approximate solutions of $\mathcal{R}_G(x_0,t_0)$ and $\mathcal{BR}_G(x_B,t_0)$ respectively, and let $U=\widetilde{U}+\overline{U}$ and $U^B=\widetilde{U}^B+\overline{U}^B$. Also let $\phi\in C^1_0(\Omega)$ be a test function where $D(x_0,t_0),\;D(x_B,t_0) \subset \Omega$. Then
\begin{align}
\label{RPres}
&R(U,\bar{D}(x_0,t_0),\phi)\nonumber\\
&=\int^{x_0+\Delta{x}}_{x_0-\Delta{x}}(U\phi)(x,t_0+\D t)dx
-\int^{x_0+\Delta{x}}_{x_0-\Delta{x}}\widetilde{U}(x, t_0^+)\phi(x, t_0)dx\nonumber\\
   &\quad+\int^{t_0+\Delta{t}}_{t_0}[f(U)\phi](x_0+\Delta{x}, t)dt
    -\int^{t_0+\Delta{t}}_{t_0}[f(U)\phi](x_0-\Delta{x}, t)dt\nonumber\\
   &\quad+O(1)\left((\Delta t)^2(\Delta
   x)+(\Delta{t})^3+(\Delta{t})^2\underset{D(x_0,t_0)}{\osc}\{\widetilde{U}\}\right)\|\phi\|_\infty,
\end{align}
\begin{align}
\label{BRPres}
&R(U^B, \bar{D}(x_B,t_0), \phi)\nonumber\\
&=\int^{x_B+\Delta{x}}_{x_B}(U^B\phi)(x, t_0+\D t)dx-\int^{x_B+\Delta{x}}_{x_B}\widetilde{U}^B(x, t_0^+)\phi(x, t_0)dx\nonumber\\
   &\quad+\int^{t_0+\Delta{t}}_{t_0}[f(U^B)\phi](x_B+\Delta{x}, t)dt-\int^{t_0+\Delta{t}}_{t_0}[f(U_B)\phi](x_B, t)dt\nonumber\\
   &\quad+O(1)\left((\Delta t)^2+(\Delta t)(\Delta x)+\Delta{t}\underset{D(x_B,t_0)}{osc}\{\widetilde{U}^B\}\right)\|\phi\|_\infty,
\end{align}
where $U_B$ given in Theorem \ref{thm2.1}, and $\underset{\Lambda}{osc}\{w\}$ is the oscillation of a function $w$ in the set $\Lambda$, and $\bar{D}(x_0,t_0)$, $\bar{D}(x_B,t_0)$ are the closures of $D(x_0,t_0)$, $D(x_B,t_0)$ in \eqref{ingrid}, \eqref{bgrid}, respectively.
\end{thrm}
\begin{proof}
The statement follows from the similar argument to the proof of Theorem 2.2 in \cite{HCHY} Appendix B.
\end{proof}
\section{Generalized Glimm scheme and wave interaction estimates}
\setcounter{equation}{0}
\label{sec3}

In this section, we introduce the generalized Glimm scheme (GGS). Also, to obtain the stability of GGS, we give the formulates of the modified version of Glimm-Goodman's wave interaction estimates.

\subsection{Generalized Glimm scheme for \eqref{1.2.2}}

In this subsection, we introduce the non-staggered generalized Glimm scheme for constructing the approximate solutions of \eqref{1.2.2}. To describe the scheme, we first partition the domain into
\begin{equation}
\label{3.0}
x_k=x_B+k\Delta x,\quad t_n=n\Delta t,\quad k, n=0,1,2,\cdots,
\end{equation}
for sufficiently small $\Delta x>0$ and $\Delta t>0$ satisfying the C-F-L condition \eqref{2.0.0.0}. The $n$-th time strip $T_n$ is given by
$$
T_{n}:=[x_B,\infty)\times[t_n,t_{n+1}),\quad n=0,1,2,\cdots.
$$
Suppose that the approximate solution, which is denoted by $U_{\theta,\Delta{x}}(x,t)$, has been constructed in $T_n$. Then, choose a random number $\theta_n \in (-1, 1)$ and define $U_k^n\equiv(\rho_k^n,m_k^n,E_k^n)$ by
$$
U_k^n:=U_{\theta,\Delta{x}}(x_{2k}+\theta_n\Delta{x},t_n^-),\ k=1,2,\cdots.
$$
To initiate the scheme, at $n=0$, we set $\theta_0$ and $t_0^-$ to be zero. The points $\{(x_{2k}+\theta_n\Delta{x},t_n^-)\}_{k=1}^{\infty}$ are called the $mesh \; points$ of the scheme, we further define the points $\{(x_0,t_n+\D t/2)\}_{n=0}^{\infty}$ as the mesh points on the boundary $x=x_B$. Then, $U_{\theta, \Delta{x}}$ in next $T_{n+1}$ are constructed by solving a set of Riemann problems $\mathcal{R}_G$ with Riemann data
\begin{align}
\label{tnstepini}
U(x,t_n)=\left\{
\begin{array}{ll}
U^n_k,&\; \mbox{if } x_{2k-1}\leq x<x_{2k},\\
U^n_{k+1},&\;\mbox{if }  x_{2k}<x\leq x_{2k+1},
\end{array}\right.\ k=1,2,\cdots.
\end{align}
and the boundary-Riemann problem $\mathcal{BR}_G$ with boundary-Riemann data near the boundary $x=x_B$:
\begin{alignat}{2}
\label{3.0.3}
\left\{\begin{array}{ll}
U(x,t_n)=U^n_1,&\mbox{if } x_0<x\leq x_1,\\
U(x_B,t)=U_B^n:=U_B(t_n), &\mbox{if }t_n\le t<t_{n+1}.
\end{array}\right.
\end{alignat}
where $U_B^n$ is chosen similar as \eqref{bdcond}. We mention that the entropy condition $(\mathcal{E})$ in Section 2
is imposed for the approximate solutions near the boundary.
\begin{fig}
\label{fig1}
\includegraphics[height = 4cm, width=10cm]{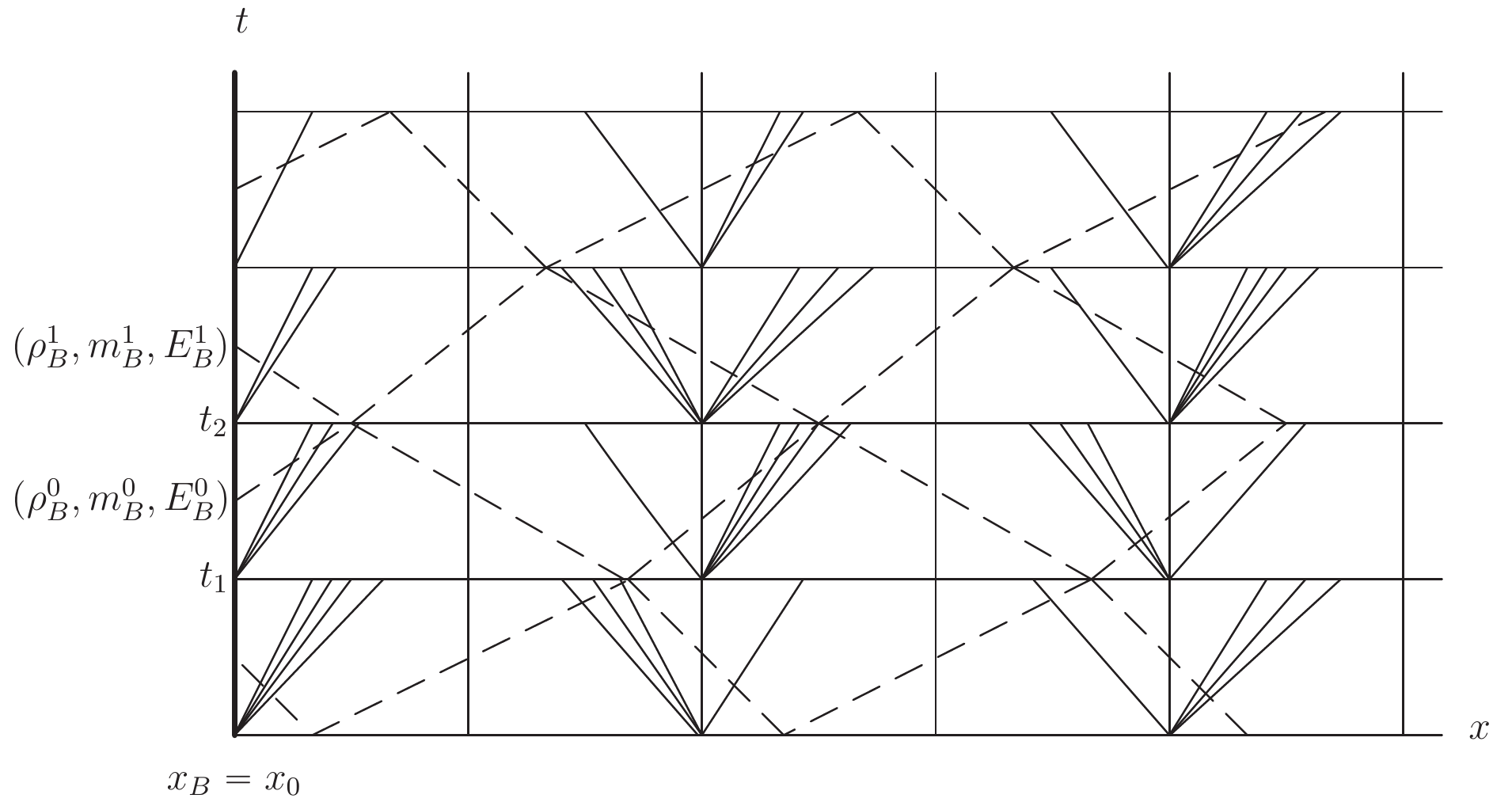}\\
Figure 1. Approximate solution of IBVP with mesh curves.\medskip
\end{fig}
Note that, in view of \eqref{2.7.1}, $U_{\theta ,\Delta{x}}$ has an explicit representation
\begin{equation}
\label{aproxsol2}
U_{\theta ,\Delta{x}}(x,t)=S(x,t,\widetilde{U}_{\th,\D x}(x,t))\widetilde{U}_{\th,\D x}(x,t),\quad (x,t) \in T_{n},
\end{equation}
where $\widetilde{U}_{\theta,\Delta{x}}$ consists of the set of weak solutions to the corresponding classical Riemann problems $\mathcal{R}_C$ and boundary-Riemann problem $\mathcal{BR}_C$ in $T_{n}$. The C-F-L condition \eqref{2.0.0.0} ensures that the classical elementary waves in each $T_{n}$ do not interact each other before time $t=t_{n+1}$. Repeating this process, we construct the approximate solution $U_{\theta,\Delta x}$ of \eqref{1.2.2} in the whole $[x_B,\infty)\times[0,\infty)$ by the generalized Glimm scheme with a random sequence $\theta:=(\theta_0, \theta_1, \theta_2,\ldots)$ in $(-1,1)$, see Figure \ref{fig1}.

To obtain the desired estimates, it is favorable to consider curves comprising line segments joining mesh points rather then horizontal lines. Therefore, we define the mesh curves for the non-local Glimm functionals introduced in \cite{GL}. A mesh curve $J$ for \eqref{1.2.2} is a piecewise linear curve that connects the mesh point $(x_{2k}+\theta_n\Delta x,t_n)$ on the left with $(x_{2k+2}+\theta_{n+1}\Delta x,t_{n+1})$ or $(x_{2k+2}+\theta_{n-1}\Delta x,t_{n-1})$ on its right, $k,n=0,1,2,\cdots$ together with the line segments joining the points $(x_B,t_n+\frac{\D t}{2})$ and $(x_B+\th_n\D x,t_n)$ and some portion of the boundary (see Figure \ref{fig1} and \cite{CHS1}). Simultaneously, the mesh curve $J$ divides the domain $[x_0,\infty)\times[0,\infty)$ into $J^+$ and $J^-$ regions such that $J^-$ contains the line $[x_0,\infty)\times\{0\}$. We can partially order two mesh curves by saying $J_2>J_1$ (or $J_2$ is a {\it successor} of $J_1$) if every mesh point of $J_2$ is either on $J_1$ or contained in $J_1^+$. In particular, $J_2$ is an {\it immediate successor} of $J_1$ if $J_2>J_1$ and all mesh points on $J_2$ except one are on $J_1$. A diamond region is a closed region enclosed by a mesh curve and its immediate successor.

\subsection{Wave interaction estimates}

In this subsection we study all kinds of nonlinear wave interactions and estimate the classical wave strengths in each time step through the wave interactions between the classical waves and the perturbations in the previous time step.\

By connecting all the mesh points by the mesh curves, the region $[x_B,\infty)\times \mathbb{R}^+$ can be decomposed as the sets of diamond, triangular and pentagon regions. The wave interactions can be divided into the following three types:
\begin{enumerate}
\item [(I)] In each diamond region, the incoming generalized waves from adjacent Riemann problems interact with each other and emerge as the outgoing generalized waves of the Riemann problem in the next time step;
\item [(II)] In each triangular region, the incoming generalized waves from the Riemann problem at the boundary interact with each other and emerge as the outgoing generalized waves of the boundary-Riemann problem
in the next time step;
\item [(III)] In each pentagonal region, two families of incoming generalized waves, one from the boundary-Riemann problem and the other from adjacent Riemann problem,
interact with each other and emerge as the outgoing generalized waves of the Riemann problem in the next time step.
\end{enumerate}
In each diamond (or triangular and pentagonal) region, all the generalized waves comprise classical outgoing waves and perturbations. Therefore, the objective of wave interaction estimates is to estimate how the wave strengths of classical outgoing waves are influenced by the interaction or reflection of generalized incoming waves.\

We start at the wave interaction estimates for type (I). Suppose $(x,t)\in (x_B,\infty)\times[0,\infty)$ and let $\mathcal{R}_G(U_R;U_L,x,t)$ denote the generalized Riemann solution of $\mathcal{R}_G(x,t)$ connecting left constant state $U_L$ with right constant state $U_R$. Also let $\mathcal{R}_C(U_R;U_L,x,t)$ be the solution of the corresponding classical Riemann problem $\mathcal{R}_C(x,t)$. Then, the {\it classical wave strength} of $\mathcal{R}_G(U_R;U_L,x,t)$ is defined as the wave strength of $\mathcal{R}_C(U_R;U_L,x,t)$ which can be written as
\begin{equation}
\label{3.1}
\ve=\varepsilon(U_R;U_L,x,t)=(\ve_1,\ve_2,\ve_3).
\end{equation}
In other words, the jump discontinuity $\{U_L,U_R\}$ is resolved into $U_L=\widetilde{U}_0$, $\widetilde{U}_1,\ \widetilde{U}_2$ and $\widetilde{U}_3=U_R$ such that $\widetilde{U}_{j+1}$ is connected to $\widetilde{U}_j$ on the right by a $j$-wave of strength $\varepsilon_j$. Note that $\mathcal{R}_C(U_R;U_L,x,t)$ is independent of the choice of $(x,t)$. We say that an $i$-wave and a $j$-wave approach if either $i>j$, or else $i = j$ and at least one wave is a shock. Given another $\mathcal{R}_G(U'_R;U'_L,x',t')$ with classical wave strength $\ve'=(\ve_1',\ve_2',\ve_3')$, the $wave$ $interaction$ $potential$ associated with $\a$, $\a'$ is defined as
\begin{equation}
\label{3.2}
D(\varepsilon,\varepsilon'):=\sum\{|\varepsilon_i\varepsilon_j'|:\varepsilon_i \mbox{ and } \varepsilon'_j \mbox{ approach}\}.
\end{equation}
Assume that $J'$ is an immediate successor of $J$. Let $\Gamma_{k,n}$ denote the diamond region centered at $(x_{2k},t_{n})$ and enclosed by $J$ and $J'$. Four vertices of $\Gamma_{k,n}$ are
\begin{align*}
   \begin{array}{ll}
     \mathpzc{N}=(x_{2k}+\theta_{n+1}\Delta x,t_{n+1}),& \mathpzc{E}=(x_{2k}+\theta_{n}\Delta x,t_{n}),\\
     \mathpzc{W}=(x_{2k+2}+\theta_{n}\Delta x,t_{n}), & \mathpzc{S}=(x_{2k}+\theta_{n-1}\Delta x, t_{n-1}),
   \end{array}
 \end{align*}
or
\begin{align*}
   \begin{array}{ll}
     \mathpzc{N}=(x_{2k}+\theta_{n+1}\Delta x,t_{n+1}),& \mathpzc{E}=(x_{2k-2}+\theta_{n}\Delta x,t_{n}),\\
     \mathpzc{W}=(x_{2k}+\theta_{n}\Delta x,t_{n}), & \mathpzc{S}=(x_{2k}+\theta_{n-1}\Delta x, t_{n-1}),
   \end{array}
 \end{align*}
see Figure \ref{fig2}. Here $\{\theta_{n-1}, \theta_{n},\theta_{n+1} \}$ are random numbers in $(-1,1)$. Define $$R(U):=[R_1(U),R_2(U),R_3(U)],$$ where $R_j$ is the right eigenvector of Jacobian matrix $Df$ associated with the eigenvalue $\lambda_j$. Note that $R(U)$ is invertible in $\Omega$. Then we have the following interaction estimate.
\begin{thrm}
\label{thm3.1}
Define $\widetilde{U}_L(x_L,t_n):=(\widetilde{\rho}_L,\widetilde{m}_L,\widetilde{E}_L)^T,\ \widetilde{U}_R(x_R,t_n):=(\widetilde{\rho}_R,\widetilde{m}_R,\widetilde{E}_R)^T$ where $x_L=x_{2k-2}+\th_n\D x, \ x_R=x_{2k+2}+\th_n\D x$. Also let
$U_M(x_M,t_n):=(\rho_M,m_M,E_M)^T$ denote the intermediate state where $x_M=x_{2k}+\th_n\D x$. Let $U_L$ and $U_R$ be the constant states in $\widetilde{\Omega}$, which can be expressed as $U_L=\widetilde{U}_L+\widebar{U}_L$, $U_R=\widetilde{U}_R+\widebar{U}_R$ where
\begin{eqnarray}
\label{aproxperturb}
&&\begin{split}
&\widebar{U}_L=(S_L-I_3)\widetilde{U}_L, \quad S_L=S(x_L,\D t,\widetilde{U}_L),\\
&\widebar{U}_R=(S_R-I_3)\widetilde{U}_R, \quad S_R=S(x_R,\D
t,\widetilde{U}_R).
\end{split}
\end{eqnarray}
Suppose that the classical wave strengths of the generalized {\rm(}in-coming{\rm)} waves across the boundaries $\mathpzc{WS}$ and $\mathpzc{SE}$ of $\Gamma_{k,n}$ are, respectively,
\begin{equation}
\label{3.3}
\a(U_M ;\widetilde{U}_L, x_L,  t_{n-1})=(\a_1,\a_2,\a_3) \quad\text{and}\quad
\beta (\widetilde{U}_R ;U_M, x_R,  t_{n-1})=(\beta_1,\beta_2,\b_3),
\end{equation}
and that the classical wave strength of the generalized {\rm(}out-going{\rm)} wave across the boundary $\mathpzc{WNE}$ is
\begin{equation}
\label{3.4}
\varepsilon(U_R;U_L,x_M,t_n)=(\varepsilon_1,\varepsilon_2,\ve_3),
\end{equation}
see Figure \ref{fig2}. Then, there exist constants $C',C''$ depending on $k,n$ such that
\begin{equation}
\label{3.6} |\varepsilon|\le(1+\z_{k,n}\D
t)(|\a|+|\beta|)+C'D(\a,\beta)+C''\D t\D x+O(1)(\D t)^3, \quad
|\a|+|\beta|\rightarrow 0,
\end{equation}
where
$$
\z_{k,n}=\frac{1}{2}\big(\frac{7-\gamma}{3}{h_1} \tilde{u}+
\frac{4}{3}
{h_2} \tilde{u}-\frac{\gamma(\gamma-1)}{\tilde{\rho}
\tilde{c}^2}\beta q\big)(x_{2k}+\th_n\D x, t_n).
$$
\end{thrm}
\begin{fig}
\label{fig2}
\includegraphics[height=4cm, width=7cm]{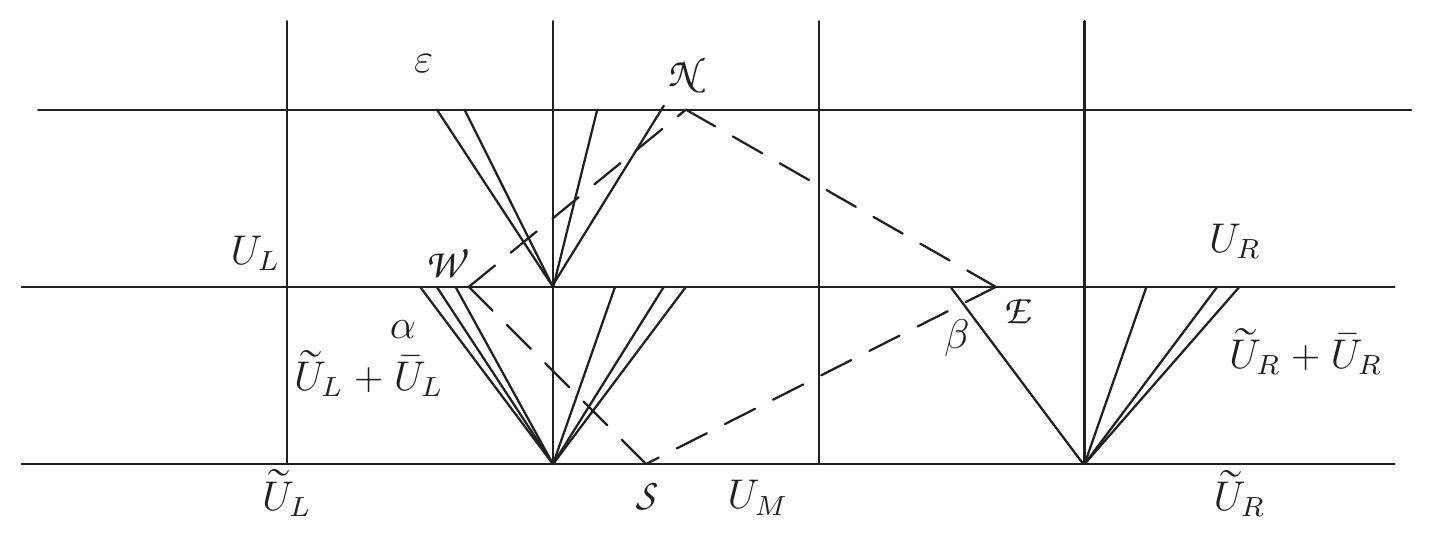}\quad
\includegraphics[height=4cm, width=7cm]{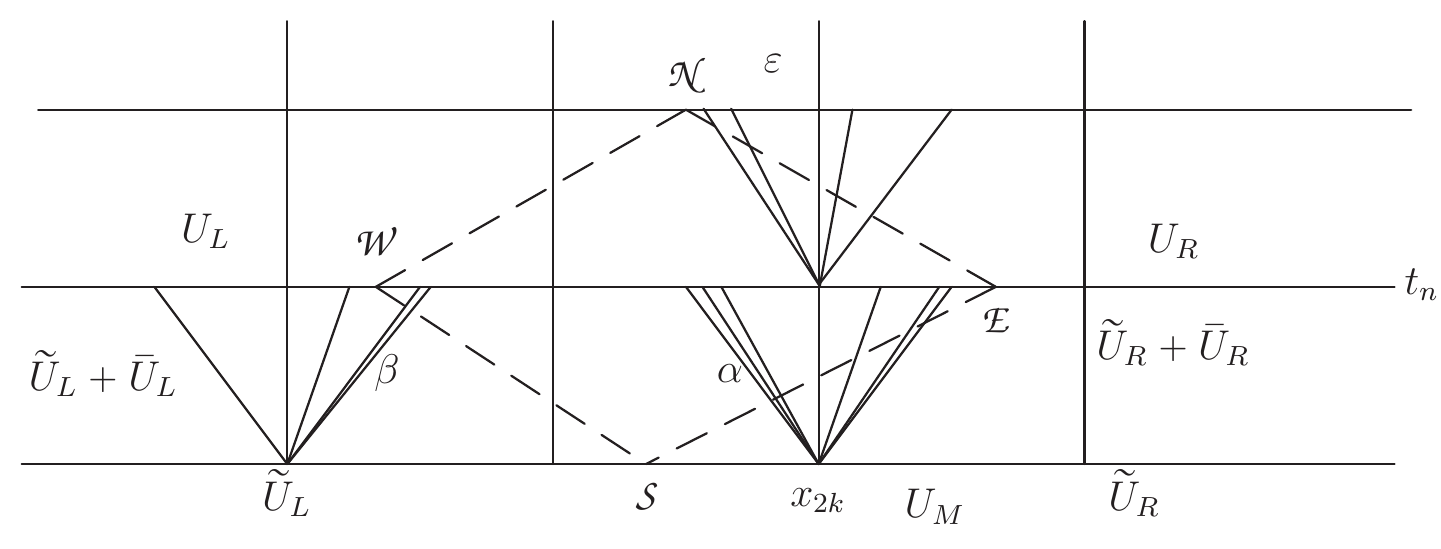}\\
Figure 2. Classical wave strengths in the diamond region $\Gamma_{k,n}$.\medskip
\end{fig}

\begin{proof}
First, following the results of \cite{SM}, we have
\begin{equation}
\label{3.7}
U_R-U_M=\overline{U}_R+\sum_{j=1}^3\beta_jR_j+ \sum_{j\leqslant
i}\beta_j\beta_i\big(R_j\cdot\nabla R_i\big)\left(1-\frac{\delta_{ij}}{2}\right)+O(1)|\beta|^3,
\end{equation}
\begin{equation}
\label{3.8}
U_L-U_M=\overline{U}_L+\sum_{j=1}^3(-\a_j)R_j+\sum_{j\geqslant
i}\a_j\a_i\big(R_j\cdot\nabla
R_i\big)\left(1-\frac{\delta_{ij}}{2}\right)+O(1)|\a|^3,
\end{equation}
where $R_j$ denotes the right eigenvector of $Df$ associated with eigenvalue $\lambda_j$, and $\delta_{ij}$ is the kronecker delta. Here, the coefficients are all evaluated at $U_M$. Similarly, we have
\begin{align}
\label{3.9}
U_R-U_L&=\sum_{i=1}^3\varepsilon_iR_i(U_L)+\sum_{j\leqslant
i}\varepsilon_j\varepsilon_i\big(R_j\cdot\nabla
R_i\big)(U_L)\left(1-\frac{\delta_{ij}}{2}\right)+O(1)|\varepsilon|^3\nonumber\\
&=\sum_{i=1}^3\varepsilon_iR_i(\widetilde{U}_L)+\sum_{j\leqslant
i}\varepsilon_j\varepsilon_i\big(R_j\cdot\nabla
R_i\big)(\widetilde{U}_L)\left(1-\frac{\delta_{ij}}{2}\right)\nonumber\\
&\quad+\left(\sum_{i=1}^3\varepsilon_i\nabla R_i(\widetilde{U}_L)\right)\widebar{U}_L+\mathcal{C}(\Delta x,|\varepsilon|),
R_i\big)\left(1-\frac{\delta_{ij}}{2}\right)+O(|\varepsilon|^3)+(\widebar{U}_R-\widebar{U}_L),
\end{align}
where $\mathcal{C}(\Delta x,|\varepsilon|)$ denotes the cubic terms of $\Delta x$ and $|\varepsilon|$. Considering $\varepsilon$ as a function of $\a,\ \beta$, we observe that $\varepsilon=0$ when $\a=\beta=S_L-S_R=0$ where $S_L,\ S_R$ are in \eqref{aproxperturb}. It implies
\begin{equation}
\label{3.9.1} \varepsilon=O(1)(\Delta x+|\a|+|\beta|).
\end{equation}
Also, by Taylor expansion of $R_i$ we have
\begin{equation}
\label{3.10}
R_i(\widetilde{U}_L)=R_i(U_M)-\sum_{j=1}^3\a_j(R_j\cdot\nabla
R_i)(U_M)+O(1)|\a|^2.
\end{equation}
It follows from \eqref{3.9}-\eqref{3.10} that
\begin{align}
\label{3.11}
U_R-U_L
&=\sum_{i=1}^3\varepsilon_iR_i(U_M)+\sum_{j\leqslant
i}\varepsilon_j\varepsilon_i(R_j\cdot\nabla
R_i)(U_M)\left(1-\frac{\delta_{ij}}{2}\right)\nonumber\\
&\quad-\sum_{\scriptstyle i,j}\varepsilon_i\a_j(R_j\cdot\nabla
R_i)(U_M)+\left(\sum_{i=1}^3(\a_i+\beta_i)\nabla
R_i(U_M)\right)\widebar{U}_L+\mathcal{C}(\Delta x,|\a|+|\beta|).
\end{align}
Next, from \eqref{2.13.1} and \eqref{aproxperturb}, we have $\widebar{U}_L=O(1)(\Delta t)=O(1)(\Delta x)$ and
\begin{align}
\label{3.14}
\widebar{U}_R-\widebar{U}_L
&=(S_R-I_3)\widetilde{U}_R-(S_L-I_3)\widetilde{U}_L \nonumber\\
&=(S_M-I_3)(\widetilde{U}_R-\widetilde{U}_L)+(S_R-S_M)\widetilde{U}_R+(S_M-S_L)\widetilde{U}_L.
\end{align}
By comparing \eqref{3.11} with \eqref{3.7}, \eqref{3.8} and using
\eqref{3.14}, we obtain
\begin{align}
\label{3.17}
\varepsilon^T &= (\a+\beta)^T+\sum_{j<i}\a_i\beta_jL_{ij}(U_M)-R^{-1}(U_M)\left(\sum_{i=1}^3(\a_i+\beta_i)\nabla R_i(U_M)\right)(S_L-I_3)\widetilde{U}_L\nonumber\\
&\quad+R^{-1}(U_M)(S_M-I_3)(\widetilde{U}_R-\widetilde{U}_L)+R^{-1}(U_M)[(S_R-S_M)\widetilde{U}_R+(S_M-S_L)\widetilde{U}_L]\nonumber\\
&\quad+\mathcal{C}(\Delta x,|\a|+|\beta|),
\end{align}
where $R=[R_1,R_2,R_3]$, $S_M=S(x_M,\D t,\widetilde{U}_M)$ and $L_{ij}:=R^{-1}(R_i\cdot \nabla R_j-R_j\cdot \nabla R_i)$.\\

Now, we rewrite all the terms in \eqref{3.17} into the terms evaluated at state $U_M$. By Taylor expansion with respect to $\D t$ and a tedious calculation, we can write $(S_L-I_3)\widetilde{U}_L$ as
\begin{align}
\label{3.17-1}
(S_L-I_3)\widetilde{U}_L
&=(S_M-I_3)U_M+(S_M-I_3)(\widetilde{U}_L-U_M)+(S_L-S_M)\widetilde{U}_L \nonumber \\
&=(S_M-I_3)U_M+O(1)\osc\{\widetilde{U}\}\D t+O(1)\D t\D x,
\end{align}
where $\osc\{\widetilde{U}\}$ is the oscillation of $\widetilde{U}$ in the diamond region. Next, to reformulate the term $(S_R-S_M)\widetilde{U}_R+(S_M-S_L)\widetilde{U}_L$ in \eqref{3.17}, we first define
$$
\mathscr{F}(x,\D t,\widetilde{U}):=S(x,\D t,\widetilde{U})\widetilde{U}.
$$
Then, we have
\begin{align}
\label{3.17-2}
&(S_R-S_M)\widetilde{U}_R+(S_M-S_L)\widetilde{U}_L \nonumber\\
&=\mathscr{F}(x_R,\D t,\widetilde{U}_R)-\mathscr{F}(x_L,\D t,\widetilde{U}_L)-S_M(\widetilde{U}_R-\widetilde{U}_L) \nonumber\\
&=\mathscr{F}_x(x_M,\D t,\widetilde{U}_M)\D x+(\mathscr{F}_U(x_M,\D t,\widetilde{U}_M)-S_M)(\widetilde{U}_R-\widetilde{U}_L) \nonumber\\
&=W(\widetilde{U}_M)\D t\D x+\Psi(\widetilde{U}_M)(\widetilde{U}_R-\widetilde{U}_L)+\mathcal{C}(\Delta x,|\a|+|\beta|),
\end{align}
where
\begin{align}
\label{3.17-3}
&W(\widetilde{U}_M)=\Big(\widetilde{m} h'_1,\widetilde{u}\widetilde{m}(h'_1+h'_2),\widetilde{m} H(\widetilde{U}) h'_1+\widetilde{u}^2 \widetilde{m}h'_2+\beta q'\Big)^T,\\
&\Psi(\widetilde{U}_M)=\mathscr{F}_U(x_M,\D t,\widetilde{U}_M)-S_M.
\end{align}
According to \eqref{3.7} and \eqref{3.8}, we have
\begin{equation}
\label{3.19}
\widetilde{U}_R-\widetilde{U}_L=R(U_M)(\a+\beta)^T+O(1)(|\a_i||\a_j|+|\beta_i||\beta_j|).
\end{equation}
Therefore, applying Taylor expansion to the terms in \eqref{3.17} and using \eqref{3.17-1}, \eqref{3.19}, we arrive at
\begin{align}
\label{3.18}
\varepsilon^T &= [I_3-\widebar{D}(U_M)+R^{-1}(U_M)(S_M-I_3)R](\a+\beta)^T+\sum_{j<i}\a_i\beta_jL_{ij}(U_M) \nonumber\\
&\quad+R^{-1}(U_M)\mathscr{F}_x(x_M,\Delta t,U_M)\D
x+R^{-1}(U_M)\Psi(U_M)(\widetilde{U}_R-\widetilde{U}_L)+\mathcal{C}(\Delta
x,|\a|+|\beta|),
\end{align}
where
\begin{align}
\label{3.18.1} \widebar{D}(U_M) :=R^{-1}\cdot[\nabla
R_1\cdot(S_M-I_3)U_M,\nabla R_2\cdot(S_M-I_3)U_M,\nabla
R_3\cdot(S_M-I_3)U_M](U_M).
\end{align}
Define
\begin{align}
\label{3.21}
\Phi(U_M):=R^{-1}(U_M)(S_M+\Psi(U_M))R(U_M)-\widebar{D}(U_M).
\end{align}
Then, by \eqref{3.19}, \eqref{3.21} we can rewrite \eqref{3.18} as
\begin{align}
\label{3.18.2} \varepsilon^T &=
\Phi(U_M)(\a+\beta)^T+\sum_{j<i}\a_i\beta_jL_{ij}(U_M)+R^{-1}(U_M)W(U_M)\D t\D
x+\mathcal{C}(\Delta x,|\a|+|\beta|).
\end{align}
\\
After a tedious calculation, we obtain the eigenvalues of $\Phi\cdot \Phi^T(U_M)$ are $\m_1=\m_2=\m_3=1+\kappa_M \Delta t$ where
\begin{equation}
\label{3.26.1} \kappa_M=\frac{7-\gamma}{3}{h_1}_M \tilde{u}_M+
\frac{4}{3} {h_2}_M
\tilde{u}_M -\frac{\gamma(\gamma-1)}{\rho_M c_M^2}\beta q_M.
\end{equation}
So, the 2-norm of $\Phi$ evaluated at $x=x_M$ is
\begin{equation}
\label{3.26} \|\Phi\|_2=\sqrt{1 + \kappa_M \Delta t}\leq 1 +
\frac{\kappa_M}{2} \Delta t.
\end{equation}
Finally, choose
\begin{equation}
\label{3.26.2} \z_{k,n}=\frac{\kappa_M}{2}, \quad C'=|Lij(U_M)|,
\quad C''=|R^{-1}(U_M)\cdot W(U_M)|,
\end{equation}
where $\kappa_M$ is in \eqref{3.26.1} and $W(U_M)$ is in \eqref{3.17-3}. Then, by \eqref{3.18.2}, \eqref{3.26} we establish \eqref{3.6}. The proof is complete.
\end{proof}

\begin{remark}
\label{rem3.2}
{\rm(}1{\rm)} When the duct is uniform {\rm(}$h_1=0${\rm)} and the effects of friction and heat {\rm(}$h_2=\beta q=0${\rm)} are ignored, inequality \eqref{3.6} is reduced to Glimm's estimate \cite{GL}:
\begin{equation}
\label{3.28.0} |\varepsilon| \leq
|\gamma|+|\beta|+C_0D(\gamma,\beta)+\mathcal{C}(\Delta
x,|\gamma|+|\beta|).
\end{equation}
{\rm(}2{\rm)} If the global velocity of gas is positive, then by assumption {\rm(}A$_3${\rm)} we have
\begin{align}
\label{3.28.2} \zeta_{k,n}<\frac{1}{2}(\frac{7-\gamma}{3}{h_1}_M
\tilde{u}_M+ \frac{4}{3}{h_2}_M \tilde{u}_M)<0
\end{align}
for all $(k,n)\in \textrm{Z}\times \textrm{Z}^+$.\\
\end{remark}

Next, we provide the wave interaction estimates for cases (II) and (III). Note that the approximate solutions of boundary Riemann problems, constructed by \eqref{2.7.1}, do not match the boundary conditions. The difference between the boundary value of approximate solutions and boundary data in each time step is given in \eqref{2.14}. We consider such difference as the wave strength of a zero-speed wave attached on the boundary $x=x_B$. Here, we extend the Goodman's type of wave interaction estimates for our approximate solution $S(t,x, \widetilde{U}_{\theta, \Delta x})\widetilde{U}_{\theta, \Delta x}$. The boundary wave interaction estimates for supersonic boundary condition is the same as in Theorem \ref{thm3.1}. Here we focus of the estimate for subsonic boundary condition case.\\

To obtain the desired estimate in the following theorem, we first give some notations. First, let $U_B^n:=\widetilde{U}_B^n+\widebar{U}_B^n$ denote the approximate solution of $\CMcal{BR}_G(x_B,t_n)$ as constructed by \eqref{2.7.1}, also define the projected boundary data in the $n$-th time strip by
$$
\widehat{U}_B^n:=(\r(x_B,t_n),m(x_B,t_n),E_B^n)^T=(\r_B^n,m_B^n,E_B^n)^T,
$$
where $E_B^n$ is the third component of $U_B^n$. Then, we treat the difference between $U_B^n$ and $\widehat{U}_B^n$ on the boundary as an $\alpha_0$ wave with wave strength
$$
\a_0:=|R^{-1}(\widetilde{U}_B^n)\cdot(U_B^n-\widehat{U}_B^n)|.
$$
Let $\beta_1$ be the 1-wave of generalized Riemann problem $\CMcal{BR}_G(x_B+2\Delta x,t_n)$. If the speed of $\beta_1$ is positive, then it will not interact all the waves in $\CMcal{BR}_G(x_B,t_n)$ in the triangular or pentagonal regions, which means that such wave interaction is trivial. So, without loss of generality, we only consider the case that the speed of $\beta_1$ is negative, meaning that the boundary data $\widehat{U}_B$ is subsonic and the speeds of 2 and 3-wave are all positive. Similarly, we define $U_B^{n+1}:=\widetilde{U}_B^{n+1}+\widebar{U}_B^{n+1}$ as the approximate solution of $\CMcal{BR}_G(x_B,t_{n+1})$. We also define
\begin{eqnarray*}
&&(\widehat{U}_B^n,U_M^n):=[(\widehat{U}_B^n,U_B^n,U_Z^n,
U_M^n)/(\a_0,\a_2,\a_3)],\\
&&(\widehat{U}_B^{n+1},U_R^{n+1}):=[(\widehat{U}_B^{n+1},{U}_B^{n+1},{U}_Z^{n+1},
U_R^{n+1})/(\varepsilon_0,\varepsilon_2,\ve_3)].
\end{eqnarray*}
It means that $(\widehat{U}_B^n,U_M^n)$ is the combined profile of $\a_0$ wave and the approximate solution of $\CMcal{BR}_G(x_B,t_n)$, and $(\widehat{U}_B^{n+1},U_R^{n+1})$ is defined similarly. Note that $(U_M^n,U_R^n):=[(U_M^n,U_R^n)/(\beta_1)]$ is the 1-wave of $\CMcal{R}_G(x_2,t_n)$ on the right of $(\widehat{U}_B^n,U_M^n)$ in $n$-th time strip, see Figure \ref{fig3}.
\begin{fig}
\label{fig3}
\includegraphics[scale=0.7]{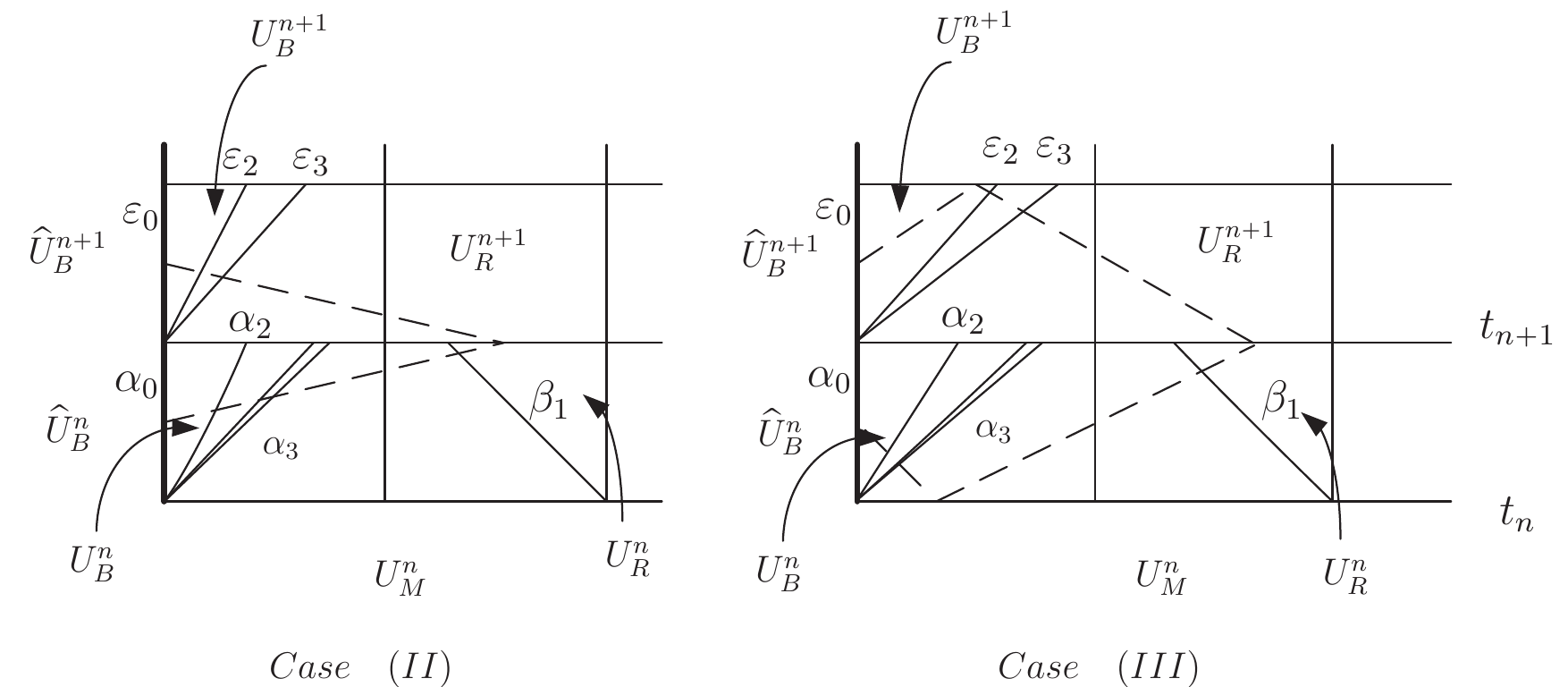}\\
Figure 3. Wave strengths in the region near the boundary $x=x_B$.\medskip
\end{fig}

We have the following wave interaction estimate for type (II) and type (III).
%

\begin{thrm}
\label{thm3.3}
{\rm(Boundary interaction estimate)} There exists a constant $C$ such that
\begin{align}
\label{4.26b}
|\varepsilon| &\leq |\a+\beta_1\mathbf{1}| +
C\Big(\sum_{App}|\a_i||\beta_1|+|\beta_1|+|\rho_B^{n+1}-\rho_B^n|+|m_B^{n+1}-m_B^n|\Big),
\end{align}
where $\mathbf{1}=(1,1,1)$.
\end{thrm}
\begin{proof}
The statement follows from the similar argument to the proof of Theorem 3.3 in \cite{G1, HS, HCHY}.
\end{proof}

Note that if $\beta_1$ is not the incoming wave of boundary triangle region, then \eqref{4.26b} is reduced by
\begin{align}
\label{4.36.1b}
|\varepsilon| &\leq |\a|+C\big(|\rho_B^{n+1}-\rho_B^n|+|m_B^{n+1}-m_B^n|\big).
\end{align}



\section{Global existence of the entropy solution}
\setcounter{equation}{0}
\label{sec4}

In this section, we prove the stability and consistency of the generalized Glimm scheme, and the entropy inequality, so that the global existence of the transonic solutions for \eqref{1.2.2} is established. The stability of the scheme, which is the core of generalized Glimm method, will be obtained by showing the non-increasing of Glimm functionals and the uniform boundedness of BV-norm of perturbation terms $\{\widebar{U}_{\th,\D x}\}$ in the approximate solutions $\{U_{\th,\D x}\}$ for \eqref{1.2.2}.

\subsection{The stability of generalized Glimm scheme}

In this subsection, we will prove the stability of the scheme, which leads to the compactness of subsequences of the approximate solutions for \eqref{1.2.2}. To start, let $U_{\theta, \Delta x}$ be the approximate solution of \eqref{1.2.2} constructed by the generalized Glimm scheme in Section 3.1. Note that $U_{\theta,\Delta x}$ can be decomposed as
\begin{equation}
\label{aproxsol3}
U_{\theta, \Delta x}=\widetilde{U}_{\theta, \Delta x}+\widebar{U}_{\theta, \Delta x},
\end{equation}
where $\widetilde{U}_{\theta, \Delta x}$ is the approximate solution by solving homogeneous conservation laws in each time step and $\widebar{U}_{\theta, \Delta x}$ is the perturbation term. We first show that $\widetilde{U}_{\theta, \Delta x}$ and its total variation are uniformly bounded. From the results of \cite{GL,SM}, it can be accomplished by showing that the Glimm functional, which will be defined later, is non-increasing in time.\\

Let $J$ be a mesh curve, $J'$ be the immediate successor of $J$ and $\Gamma_{k,n}$ be the diamond region with center $(x_{2k},t_n)$ and enclosed by $J$ and $J'$. Then, the Glimm functional $F$ for $\widetilde{U}_{\theta ,\Delta x}$ over $J$ is defined as
\begin{equation}
\label{glimfunl}
F(J):=L(J)+KQ(J),
\end{equation}
where $K$ is a positive constant which will be decided later and
\begin{align*}
L(J)&:=\sum  \{ |\a_i| : \a_i \  \mbox{crosses} \ J \}+K_1\Big(|\beta_1|+\sum_{k\in B(J)}l_B^k
 \Big),\\
Q(J)&:= \sum \{ |\a_i||\a_{i'}|: \a_i,\a_{i'} \ \mbox{cross} \  J \ \mbox{and approach} \},\\
l_B^n&:=\left\{\begin{array}{ll}
|\rho_B^{n+1}-\rho_B^n|+|m_B^{n+1}-m_B^n|, & \text{(subsonic case)} \\
|\rho_B^{n+1}-\rho_B^n|+|m_B^{n+1}-m_B^n|+|E_B^{n+1}-E_B^n|, & \text{(supersonic case)}.
\end{array}\right.
\end{align*}
Here both constants $K>1$ and $K_1>1$ will be decided later, $B(J):=\{n:P_{x_B,n}=(x_B,t_n+\Delta t/2)\in J\}$, $l_b^n$ is evaluated at the mesh point $P_{x_B,n}$, and the presence of $|\beta_1|$ depends on $\beta_1$ crosses $J$ and locates in some boundary triangle region, see Figure \ref{fig3}.\\

First, we recall the domain $\Omega$ for some constant $r,r^*$ in \eqref{1.2.3} and the case that $J$, $J'$ differ in a diamond region away from boundary. From Theorem \ref{thm3.1}, let $Q(\Gamma_{k,n}):=D(\a,\beta)$ be the wave interaction potential associated with $\a$ and $\beta$ and let
$$
C_1=\max\limits_{U\in\Omega}\Big|\sum\limits_{j<i}L_{ij}(U)\Big|\ge C'_{k,n},\ \forall\,k,n.
$$
By condition $(A_2)$, we have $u_0(x)>0,\ \forall\,x\in[x_B,\infty)$. Following \eqref{3.6}, we obtain
 \begin{align}
\label{3.32}
L(J')-L(J)&\leq C_1Q(\Gamma_{k,n})-\l_*^{-1}\z_{k,n}(|\a|+|\beta|)\D x
+\l_*^{-1}C''_{k,n}(\Delta x)^2+O(1)(\Delta x)^3,\\
\label{3.33}
Q(J')-Q(J)&\leq -Q(\Gamma_{k,n})+L(J)[C_1Q(\Gamma_{k,n})-\l_*^{-1}\z_{k,n}(|\a|+|\beta|)\Delta{x}\nonumber\\
&\quad+\l_*^{-1}C''_{k,n}(\Delta x)^2+O(1)(\Delta x)^3],
\end{align}
where $C''_{k,n},\z_{k,n}$ as in \eqref{3.6} which defined in Theorem \ref{thm3.1}. Therefore, by \eqref{glimfunl}, \eqref{3.32} and \eqref{3.33}, we have
\begin{align}
\label{3.34}
F(J')-F(J)&\leq -[K-C_1-KC_1L(J)]Q(\Gamma_{k,n})-\l_*^{-1}\z_{k,n}(|\a|+|\beta|)\D x\nonumber\\
&\quad+[1+KL(J)]\l_*^{-1}C''_{k,n}(\Delta x)^2+O(1)(\Delta x)^3.
\end{align}
If $K$ satisfies $2C_1<K\le\e/L(J)$ for some $0<\e<1/2$, then we have
\begin{align}
\label{3.38.c}
F(J)=L(J)+KQ(J)\leq L(J) + K L^2(J)\leq(1+\e)L(J),
\end{align}
and \eqref{3.34} gives an estimate
\begin{align}
\label{3.36}
F(J')&<F(J)-\l_*^{-1}\z_{k,n}(|\a|+|\beta|)\D x+\l_*^{-1}(1+\e)C''_{k,n}(\Delta x)^2+O(1)(\Delta x)^3,
\end{align}
By \eqref{3.38.c} and adding up recursive relation \eqref{3.36} over all $k$, use the condition $(A_2)$ that $\int_{x_B}^{\infty}h'_i dx\leq a^*$, i=1,2 and $\int_{x_B}^{\infty} q' dx\leq a^*$, we obtain
\begin{align}
\label{3.38}
F(J_{2})&< F(J_1)-\l_*^{-1}C_2L(J_1)\Delta x+\l_*^{-1}(1+\e)a^*C_3\Delta x+O(1)(\Delta x)^2,
\end{align}
where $C_2$ is give in \eqref{3.28.2}, $C_3$ is give in \eqref{3.17-3} Thus, if $\D x$ sufficiently small, and \eqref{3.36}, we have
\begin{align}
\label{3.38.a}
F(J_3)&\leq F(J_{2})-\l_*^{-1}C_2L(J_2)\Delta x+\l_*^{-1}(1+\e)a^*C_3\Delta x+O(1)(\Delta x)^2 \nonumber\\
      &\leq (1-\frac{C_2}{1+\e}\l_*^{-1}\Delta x)F(J_{2})+\l_*^{-1}(1+\e)a^*C_3\Delta x+O(1)(\Delta x)^2 .
\end{align}
According to \eqref{3.38} and \eqref{3.38.a}, we obtain
\begin{align}
\label{3.38.e}
F(J_{3})&\le \Big(1-\l_*^{-1}\frac{C_2}{1+\e}\D x\Big)^2F(J_1)+\Big(1-\l_*^{-1}\frac{C_2}{1+\e}\D x\Big)(\l_*^{-1}(1+\e)a^*C_3\Delta x+O(1)(\Delta x)^2)\nonumber\\
&\quad+\l_*^{-1}(1+\e)a^*C_3\Delta x+O(1)(\Delta x)^2.
\end{align}
According to the similarly argument in the previous step, we further obtain
\begin{align}
\label{3.41}
F(J_{n})&\leq \Big(1-\l_*^{-1}\frac{C_2}{1+\e}\D x\Big)^{n-1}F(J_1)+\l_*^{-1}(1+\e)a^*C_3\Delta x\sum_{k=1}^{n-1}\Big(1-\l_*^{-1}\frac{C_2}{1+\e}\D x\Big)^{k-1} \nonumber\\
&\quad+O(1)(\Delta x)^2.
\end{align}
and thus, we get
\begin{align}
\label{3.42}
F(J_{n})&\le F(J_1)+(1+\e)^2\frac{C_3}{C_2}a^*+O(1)(\Delta x)^2.
\end{align}
and $u(x,t_n)>0,\ \forall\,x\in[x_B,\infty)$. Note that in the second inequality of \eqref{3.42}, it leads to
\begin{align}
\label{3.43}
\TV_{J}\{\widetilde{U}_{\theta,\Delta x}\}&\leq O(1)L(J)\leq O(1)F(J)\nonumber\\
&\leq (1+\e)\TV\{U_0(x)\}+(1+\e)^2\frac{C_3}{C_2}a^*+O(1)(\Delta x)^2
\end{align}
for $J_k\leq J<J_{k+1}$, $k=1,\ldots,n-1$.
\begin{align}
\label{3.43.2}
\mathcal{C}:=\frac{C_3}{C_2}a^*.
\end{align}
Next, we consider the case that $J'$ is an immediate successor of $J$ so that they only differ on boundary $P_{x_B,{n}}$. For the subsonic boundary case, following condition ($A_2$), \eqref{4.26b} and \eqref{glimfunl}, we obtain
\begin{align}
\label{4.38b}
F(J')-F(J)
&=|\varepsilon|-|\a|-|\beta_1|-K_1(|\beta_1|+l_B^k)\nonumber\\
&\quad+K|\a|(|\varepsilon|-|\a|-|\beta_1|)-K(|\a_0\beta_1|+|\a_2\beta_1|+|\a_3\b_1|),\nonumber\\
&\leq O(1)C(|\a_0\beta_1|+|\a_2\beta_1|+|\a_3\beta_1|+|\beta_1|+l_B^k)-K_1(|\beta_1|+l_B^k)\nonumber\\
&\quad+O(1)CK|\a|(|\a_0\beta_1|+|\a_2\beta_1|+|\a_3\beta_1|+|\beta_1|+l^k_B)\nonumber\\
&\quad-K(|\a_0\beta_1|+|\a_2\beta_1|+|\a_3\beta_1|)+O(\Delta x)^2\nonumber\\
&\leq (-K_1+O(1)C+O(1)CK\cdot F(J))(|\beta_1|+l_B^k)\nonumber\\
&\quad+(-K+O(1)C+O(1)CK\cdot F(J))(|\a_0||\beta_1|+|\a_2||\beta_1|+|\a_3||\beta_1|)\nonumber\\
&\quad+O(1)(\Delta x)^2\leq O(1)(\Delta x)^2
\end{align}
provided constants $K_1$, $K\geq O(1)2C$, and $ KL(J)\leq\e$. The supersonic boundary case can be estimate in similar way. Now, let $J_n$ be the mesh curve located on the time strip $T_n:=(x_B,\infty)\times[t_{n-1},t_n)$ and include the half-ray $\{x=r_b,\;t\geq t_n+\Delta t/2\}$. Also, let $\TV\{U_0(x)\}:=\TV\{\r_0(x)\}+\TV\{m_0(x)\}+\TV\{E_0(x)\}$. If $\Delta x$ and $\TV\{U_0(x)\}$ are sufficiently small, then we have
\begin{align}
\label{4.41b}
F(J_{k+1})\le F(J_{k})-\l_*^{-1}\frac{C_2}{1+\e+\e^2}(\D x)F(J_k)+O(1)(\Delta x)^2,\quad k=1,...,n.
\end{align}
Therefore, based on \eqref{4.41b} and similarly step as away from boundary, we obtain $\widetilde{U}_{\theta,\Delta x}$ is defined for $t > 0$ and $\Delta x \rightarrow 0$.

Next, we verify the total variation of the perturbation for any fixed time step is also bounded. Let us denote $S_k:=S(x_k,t,\widetilde{U}_k)$, then
\begin{eqnarray*}
\TV\{\widebar{U}\} &=& \sum_k|\widebar{U}(x_{k+1})-\widebar{U}(x_{k-1})|\le\sum_k|[(S-I_3)\widetilde{U}](x_{k+1})-[(S-I_3)\widetilde{U}](x_{k-1})| \\
                   &\le& \sum_k|(S_z-I_3)(\widetilde{U}_{k+1}-\widetilde{U}_{k-1})|+\sum_k|(S_{k+1}-S_z)\widetilde{U}_{k+1}
                         +(S_z-S_{k-1})\widetilde{U}_{k-1}|
\end{eqnarray*}
According to \eqref{3.17-2}, we obtain
\begin{eqnarray}
\label{3.45}
\TV\{\widebar{U}\} &\le& \|S_z-I_3\|\sum_k\osc\{\widetilde{U}\}+2\sum_k|W(\widetilde{U})|\D x\D t+2\|\Psi\|\sum_k\osc\{\widetilde{U}\} \nonumber\\
                    &\le& \|S_z-I_3\|\TV\{\widetilde{U}\}+ O(1)a^*\D t+\|\Psi\|\TV\{\widetilde{U}\}
\end{eqnarray}

Hence the total variation of the perturbation $\widebar{U}$ is bounded by boundedness of the total variation of $\widetilde{U}$. Because of the boundedness of the total variation for approximate solutions, the constant $\mathcal{C}$ in \eqref{3.43.2} can be easily determined by the initial-boundary data, friction and heating. By \eqref{3.43}, \eqref{3.45} and the results in \cite{DH,SM}, we obtain the following theorem.

\begin{thrm}
\label{thm4.1}
For fixed $K,\ \e$ as chosen above. Let $U_{\theta,\Delta x}$ be an approximate solution of \eqref{1.2.2}{\rm} by the generalized Glimm scheme. Then under conditions {\rm(}$A_1${\rm)}-{\rm(}$A_4${\rm)}, for any given constant state $\widecheck{U}$ there exist positive constants $d$, depending on the radius $r$ of $\Omega$, such that if
\begin{align}
\label{3.47.1}
\sup_{x \in [x_B,\infty)}|U_0(x)-\widecheck{U}|\leq\frac{r}{2}, \quad\TV\{U_0(x)\}\leq d,
\end{align}
plus the condition
\begin{align}
\label{3.47.2}
\sup_{t\in{\mathbb{R}^+}}|m_B(t)-\widecheck{m}|\leq\frac{r}{2}+(1+\e)^2\mathcal{C}
\end{align}
hold for \eqref{1.2.2} with $\mathcal{C}$ as the constant in \eqref{3.43.2}. Then $U_{\theta,\Delta x}(x,t)$ is well-defined for $t \geq 0$ and sufficiently small $\Delta x>0$. Furthermore, $U_{\th,\D x}(x,t)$ has uniform total variation bound and satisfies the following properties:
\begin{enumerate}
\item [{\rm(}i{\rm)}] $\displaystyle \|U_{\theta,\Delta{x}}-\widecheck{U}\|_{L^\infty}\leq r+(1+\e)^2\mathcal{C}$.
\item [{\rm(}ii{\rm)}] $\TV\{U_{\theta,\Delta{x}}(\cdot,t)\}\leq \dfrac{r}{2}+(1+\e)^2\mathcal{C}$.
\item [{\rm(}iii{\rm)}] $\displaystyle\int_{x_B}^{\infty}|U_{\theta,\Delta{x}}(x,t_2)-U_{\theta,\Delta{x}}(x,t_1)|dx\leq
O(1)(|t_2-t_1|+\Delta{t})$.
\end{enumerate}
\end{thrm}
\begin{proof}
For (i) and (ii), note that
\begin{align}
\label{3.44}
\sup_{J_n}|U_{\theta,\Delta x}-\widecheck{U}|\leq \sup |U_0(x)-U_*|+\TV_{J_n}\{U_{\theta,\Delta x}\}.
\end{align}
We choose fixed $d$ such that $(1+\e)d\le r/2$, it follows from \eqref{3.43} that if
$$
\sup|U_0(x)-\widecheck{U}|\leq \frac{r}{2},\quad\TV\{U_0(x)\}\leq d,
$$
then, for sufficiently small $\Delta x$, we have
$$
\TV_{J_n}\{U_{\th,\D x}\}\le(1+\e)\TV\{U_0(x)\}+(1+\e)^2\mathcal{C}\le\frac{r}{2}+(1+\e)^2\mathcal{C},
$$
where $\mathcal{C}$ is the constant as in \eqref{3.43.2}. Thus,
\begin{align}
\label{3.44.1}
\sup_{\scriptstyle J_n}|\widetilde{U}_{\th,\Delta x}-\widecheck{U}|\leq r+(1+\e)^2\mathcal{C}.
\end{align}
With the above choice of $K$ and $d$, we obtain that, for sufficiently small $\Delta x$, $U_{\th,\Delta x}(x,t)$ is defined on $[x_B,\infty)\times[0,\infty)$ when ($A_1$)-($A_4$) hold. In addition, $U_{\th,\Delta x}(x,t)$ and its total variation are uniformly bounded and independent of $\Delta x$.

For (iii), without loss of generality, let $t_2>t_1$, $t_0=\sup\{t\le t_1\mid t=n\D t\text{ for some }n\}$, and let $\ell=\lfloor(t_2-t_0)/\D t\rfloor+1$. According to \eqref{2.13.1}, we have that
\begin{align*}
|U_{\th,\D x}(x,t_2)-U_{\th,\D x}(x,t_1)|
&\le |U_{\th,\D x}(y,t_0)-U_{\th,\D x}(x,t_0)|+|(S(y,U_{\th,\D x}(y,t_0))^{\ell}-I)U_{\th,\D x}(y,t_0)| \\
&=|U_{\th,\D x}(y,t_0)-U_{\th,\D x}(x,t_0)|+O(1)(\D t)
\end{align*}
for some $y\in[x-\ell\D x,x+\ell\D x]$. Hence following Corollary 19.8 in \cite{SM}, we get (iii).
\end{proof}

Therefore, by Theorem \ref{thm4.1} and Oleinik's analysis in \cite{SM}, we have the following theorem for the compactness of the subsequence of $\{U_{\theta,\Delta x}\}$.

\begin{thrm}
\label{thm.4.2}
Assume that {\rm(}$A_1${\rm)}-{\rm(}$A_4${\rm)} hold. Let $\{U_{\theta ,\Delta x}\}$ be a family of approximate solutions \eqref{1.2.2} by the GGS. Then there exist a subsequence
$\{U_{\theta ,\Delta x_i}\}$ of $\{U_{\theta ,\Delta x}\}$ and a measurable function $U$ such that
\begin{enumerate}
\item [{\rm(}i{\rm)}] $U_{\theta ,\Delta x_i}(x,t)\rightarrow U(x,t)$ in $L^1_{loc}$ as
$\Delta x_i\rightarrow 0$;
\item [{\rm(}ii{\rm)}] for any continuous function $f$, we have $f(x,t,U_{\theta,\Delta x_i})\rightarrow f(x,t,U)$
in $L^1_{loc}$ as $\Delta x_i \rightarrow 0$.
\end{enumerate}
\end{thrm}


Finally, we prove the global existence of entropy solutions to \eqref{1.2.2} by showing the consistency of the scheme and the entropy inequalities for weak solutions. By Theorem \ref{thm2.2} and the similar proof as in Theorem 3.6 of \cite{HCHY}, we obtain the consistency of our scheme and entropy inequalities, and which leads to the global existence results of Main Theorem.

\section{Numerical examples}
\setcounter{equation}{0}
\label{sec5}

In this section we present numerical examples for the one-dimensional nozzle flow equations with friction and heat source terms. The computational domain is set as $[1,10]$, the CFL number is taken as 5. We verify our theoretical result by showing the solution profile at $t=2$. The adiabatic index is $\g=1.4$, and the initial and the boundary densities is fixed as
$$
\r_0(x)=\left\{\begin{array}{ll}
1.1, & 1\le x\le 4 \\ 1, & 4<x\le 10,
\end{array}\right.\qquad \r_B(t)\equiv 1.1.
$$
For simplicity, we set $P_0(x)=\r_0(x),\ P_B(t)=\r_B(t)$, then the initial sound speed is fixed as $c_0(x)\equiv\sqrt{1.4}$. For the initial and the boundary velocity, we divided into two cases, subsonic-supersonic case and supersonic-subsonic case respectively. We set the initial and boundary velocities as
\begin{equation}
\label{subsup}
u_0(x)=\left\{\begin{array}{ll}
\sqrt{1.2}, & 1\le x\le 4 \\ \sqrt{1.6}, & 4<x\le 10,
\end{array}\right.\qquad u_B(t)\equiv\sqrt{1.2},
\end{equation}
or
\begin{equation}
\label{supsub}
u_0(x)=\left\{\begin{array}{ll}
\sqrt{1.6}, & 1\le x\le 4 \\ \sqrt{1.2}, & 4<x\le 10,
\end{array}\right.\qquad u_B(t)\equiv\sqrt{1.6}.
\end{equation}
Differ from our theoretical setting above, we let the heat profile as follows in our numerical simulation:
$$
q(x,t)=\frac{1}{2(t+1)}e^{-(x-5)^2}.
$$
The above setting will be used for the simulations of the examples below. The numerical examples is simulated by using MacCormak scheme to solve the homogeneous Riemann (or boundary-Riemann) problems and then multiplied by the contraction matrix \eqref{solver1} in the homogeneous solutions for each time steps.

\begin{exam}
\label{exam5.1}
{\rm(Contraction-Expansion nozzle in subsonic-supersonic case)\\
In this example, we consider the cross section area $a(x)$ is a $C^0$ function as follows
$$
a(x)=\frac{(x-5)^2}{160}+10,\quad x\in[1,10],
$$
and $a'(x)\equiv 0$ when $x\in\mathbb{R}\backslash[1,10]$. Then $h_1(x)=-\frac{2(x-5)}{x^2-10x+1625},\quad h_2(x)=-\frac{4\sqrt{10}\a}{\sqrt{x^2-10x+1625}}$.
Using the aforementioned setting, we have
$$
\max_{x\in[1,10]}\Big(\frac{7-\gamma}{3}h_1(x)u_0(x)+ \frac{4}{3}h_2(x)u_0(x) -\frac{\gamma(\gamma-1)}{\rho_0(x)c^2_0(x)}\beta q(x,0)\Big)<0
$$
for $\a>0.022028,\ \b>0$. The solution profile at $t=2$ with $\a=0.1,\ \b=1$ in grid sizes $\D x=0.1,\ 0.05,\ 0.01,\ 0.005$ for the density, velocity, energy and the Mach number are shown in Figure \ref{fig4}.
\begin{fig}
\label{fig4}
\includegraphics[scale=0.265]{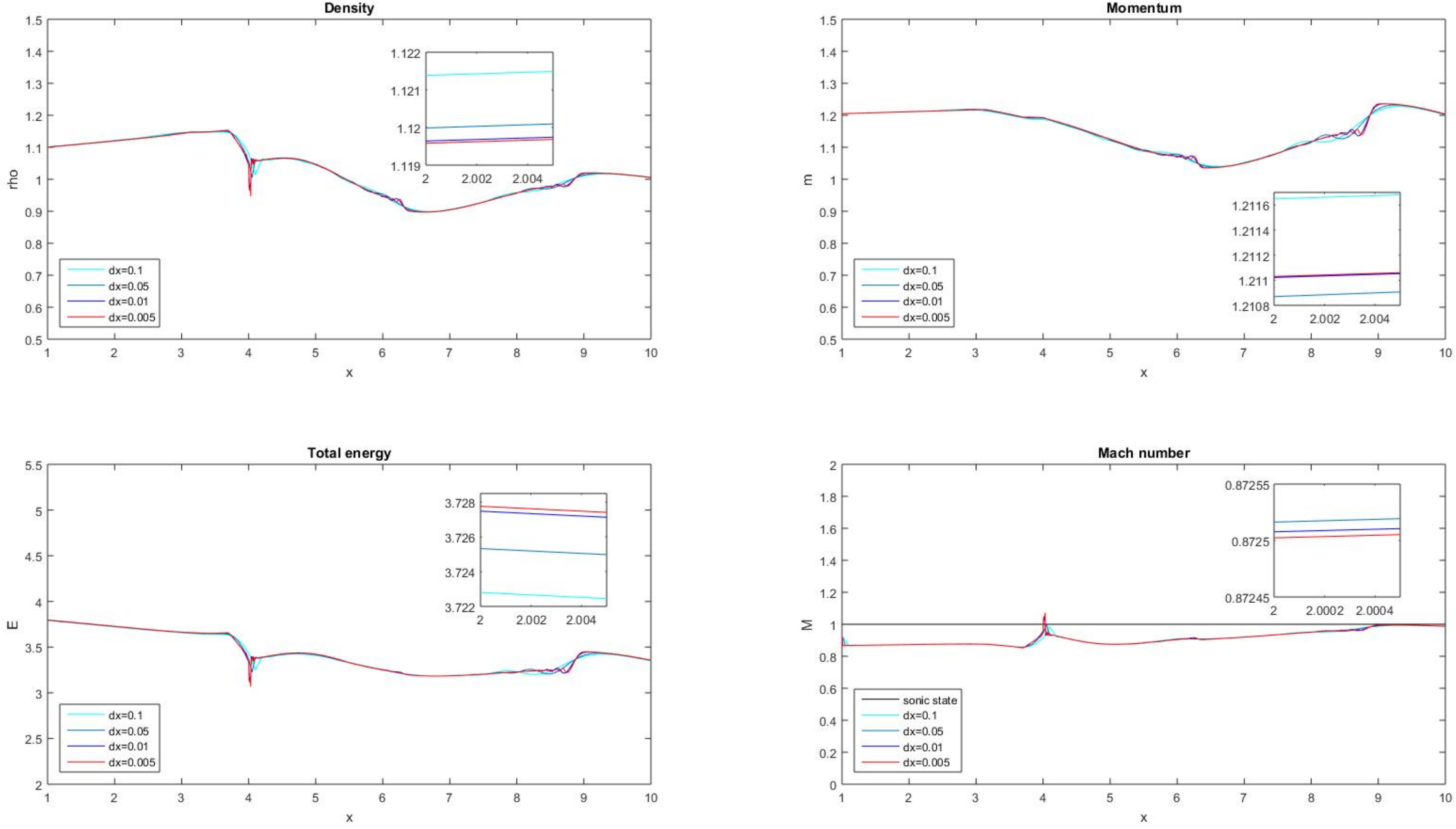}\\
Figure 4. The numerical solutions for \eqref{1.2.2} in subsonic-supersonic case.\medskip
\end{fig}
Figure \ref{fig5} displays the friction effect on the solution, we fix the grid size $\D x=0.01$ and the heat parameter $\b=1$ and set the friction parameter $\a=0.5,\ 0.1,\ 0.05$. It shows us that the greater friction causes the smaller momentum and Mach number, and thus the density distribution is pulled towards more left direction due to the contraction-expansion shape of the nozzle.
\begin{fig}
\label{fig5}
\includegraphics[scale=0.265]{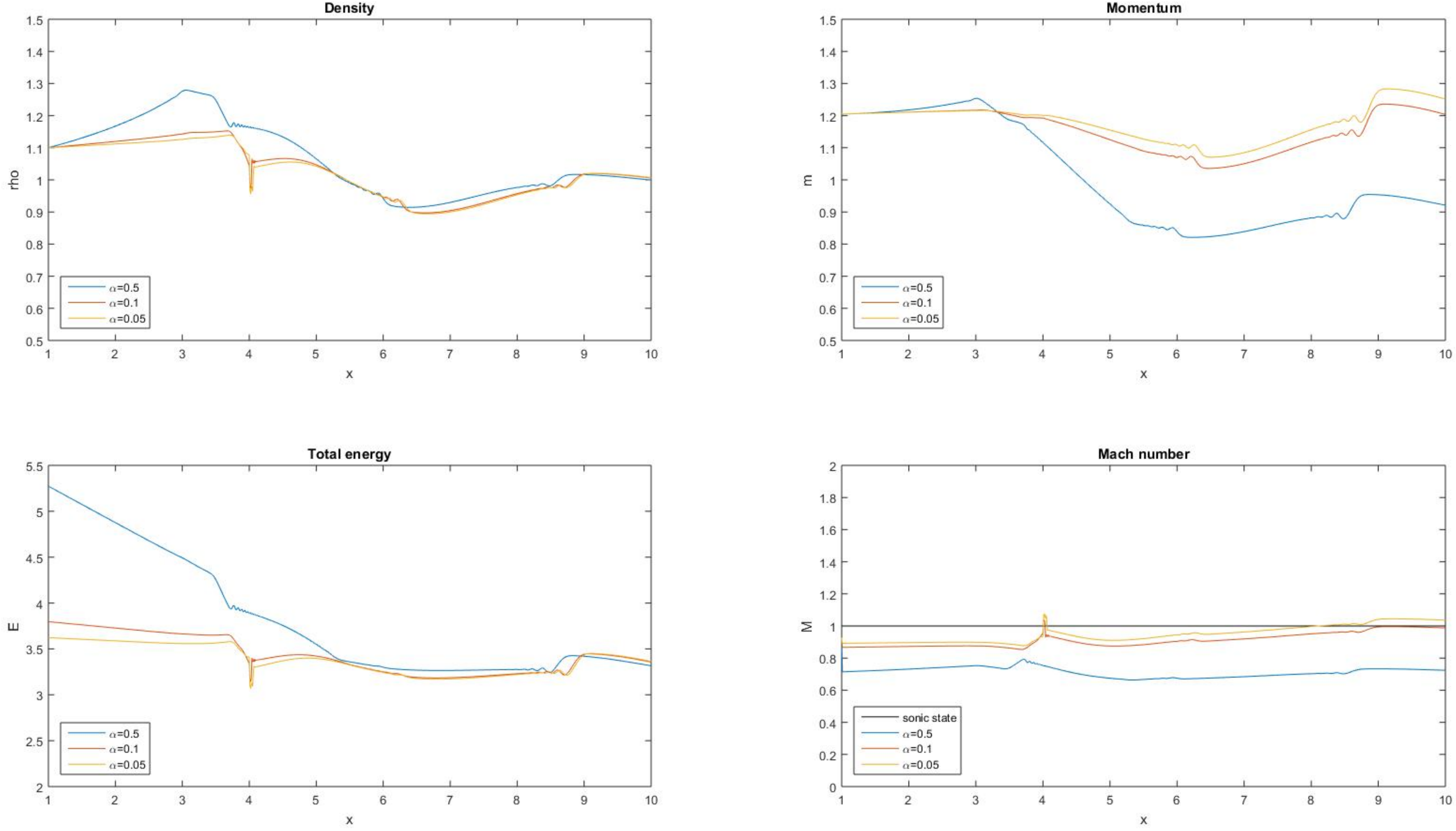}\\
Figure 5. The solutions for \eqref{1.2.2} in supersonic-subsonic case with different $\a$.\medskip
\end{fig}
In Figure \ref{fig6}, we investigate the heat effect on the solution. Fix the grid size $\D x=0.01$ and the friction parameter $\a=0.1$ and set the heat parameter $\b=4,\ 1,\ 0.25$. We see that the more heat effect at the center of the nozzle causes the lower momentum but the larger total energy, which leads the larger pressure at there, the density distribution is pushed away from the center of the nozzle.
\begin{fig}
\label{fig6}
\includegraphics[scale=0.265]{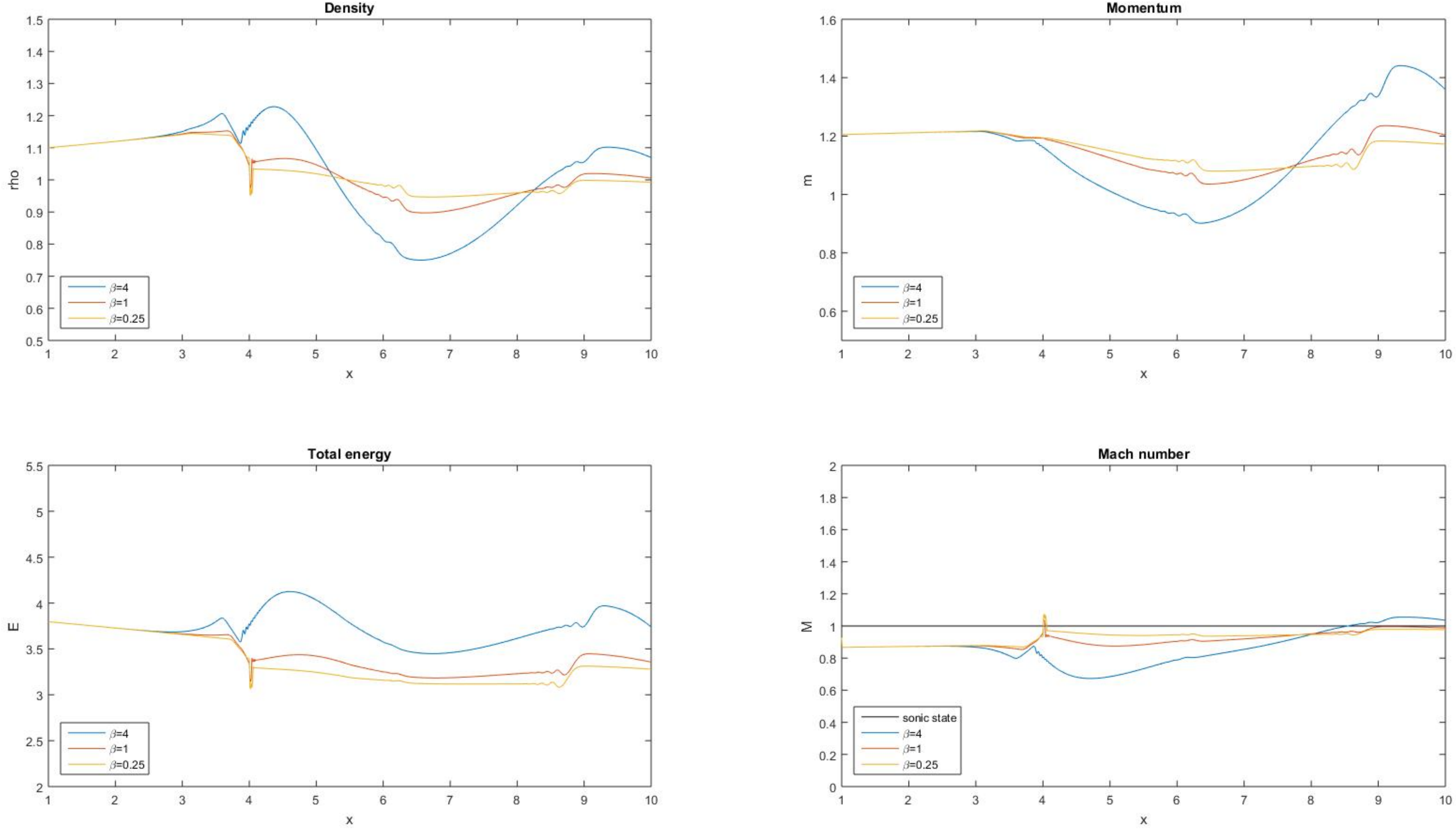}\\
Figure 6. The solutions for \eqref{1.2.2} in subsonic-supersonic case with different $\b$.\medskip
\end{fig}
}
\end{exam}

\begin{exam}
\label{exam5.2}
{\rm(Contraction-Expansion nozzle in supersonic-subsonic case)\\
The setting is the same as the Example 5.1 but impose the initial velocity as \eqref{supsub}. The solution profile at $t=2$ with $\a=0.1,\ \b=1$ in grid sizes $\D x=0.1,\ 0.05,\ 0.01,\ 0.005$ are shown in Figure \ref{fig7}. According to the given initial data, the initial shock is located at $x=4$. Figure \ref{fig7} displays that the shock front passes through the nozzle which located at $x=5$ when $t=2$. This gives us the numerical evidence that the solution exists globally under the conditions ($A_1$)-($A_4$).
\begin{fig}
\label{fig7}
\includegraphics[scale=0.265]{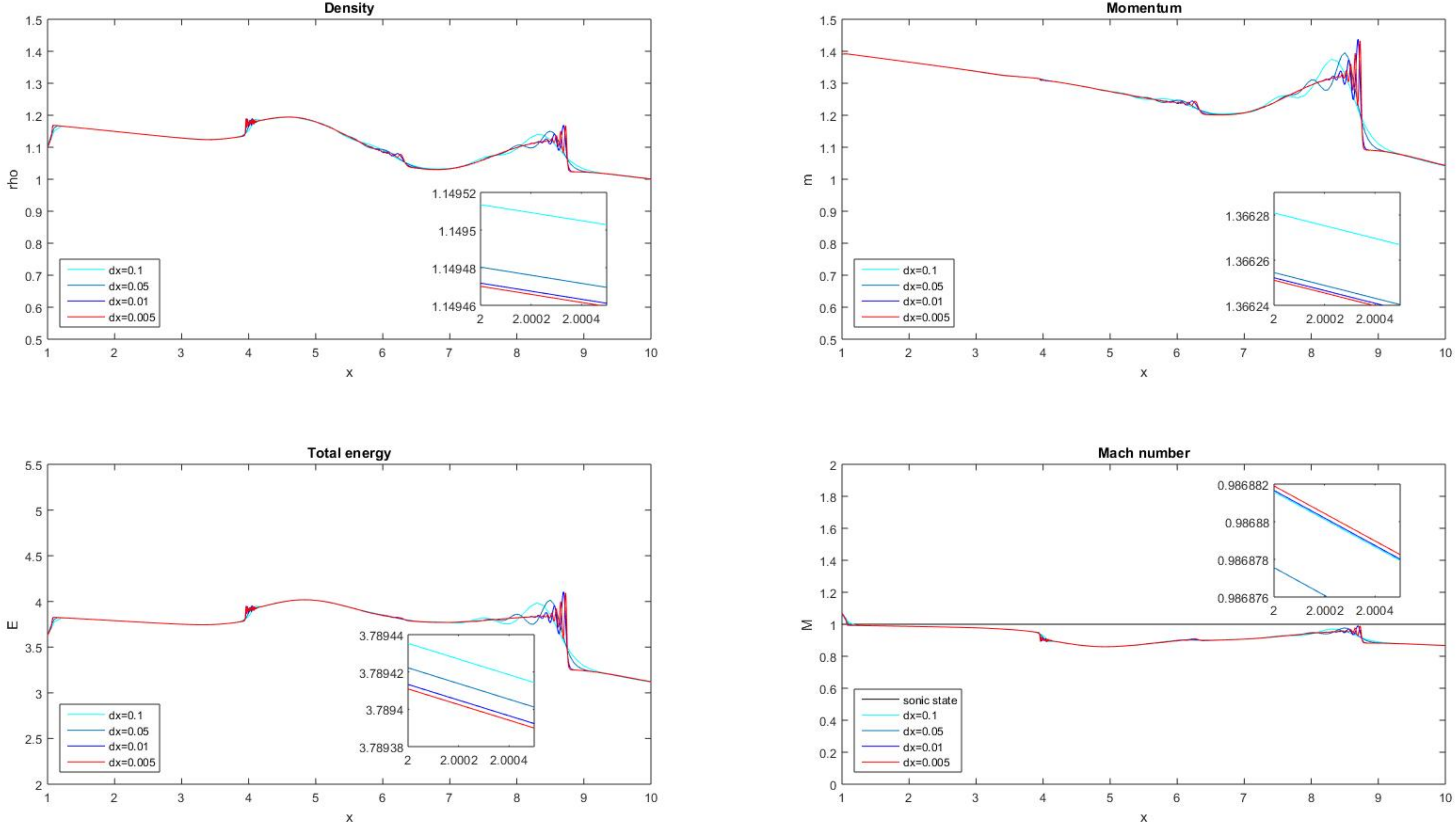}\\
Figure 7. The numerical solutions for \eqref{1.2.2} in supersonic-subsonic case.\medskip
\end{fig}
The friction effect for supersonic-subsonic case are shown in Figure \ref{fig8}. As in Figure \ref{fig5}, the greater friction causes the smaller momentum and Mach number. and the density distribution pulled towards more left direction. We also see that the greater friction cause the slower wave propagation form the location of the shock front.
\begin{fig}
\label{fig8}
\includegraphics[scale=0.265]{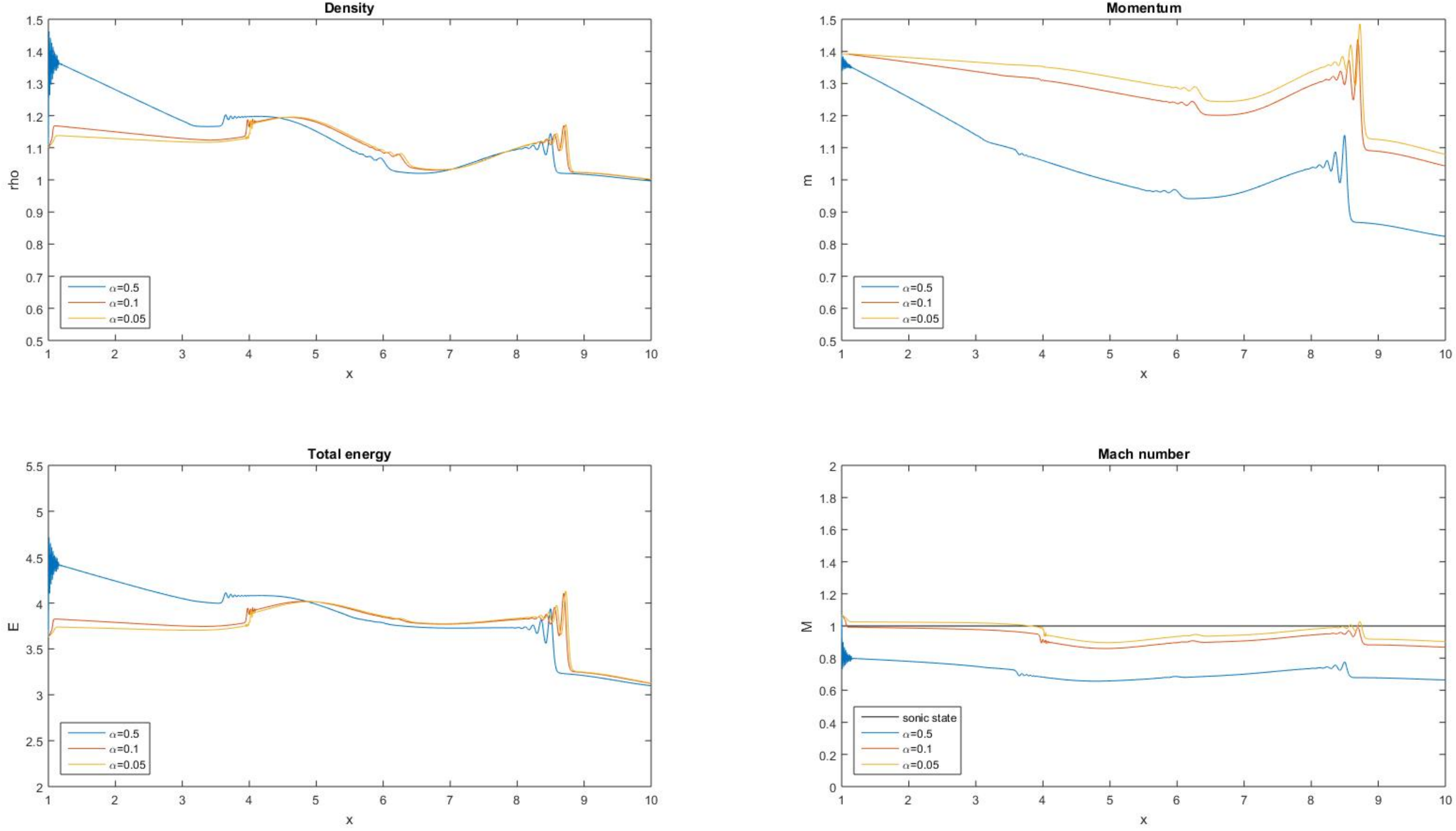}\\
Figure 8. The solutions for \eqref{1.2.2} in supersonic-subsonic case with  different $\a$.\medskip
\end{fig}
The heat effect for supersonic-subsonic case are shown in Figure \ref{fig9}. As in Figure \ref{fig6}, the more heat effect at the center of the nozzle causes the lower momentum but the larger total energy, and the density distribution is pushing away from the center of the nozzle. We also see that the greater heat cause the faster wave propagation form the location of the shock front.
\begin{fig}
\label{fig9}
\includegraphics[scale=0.265]{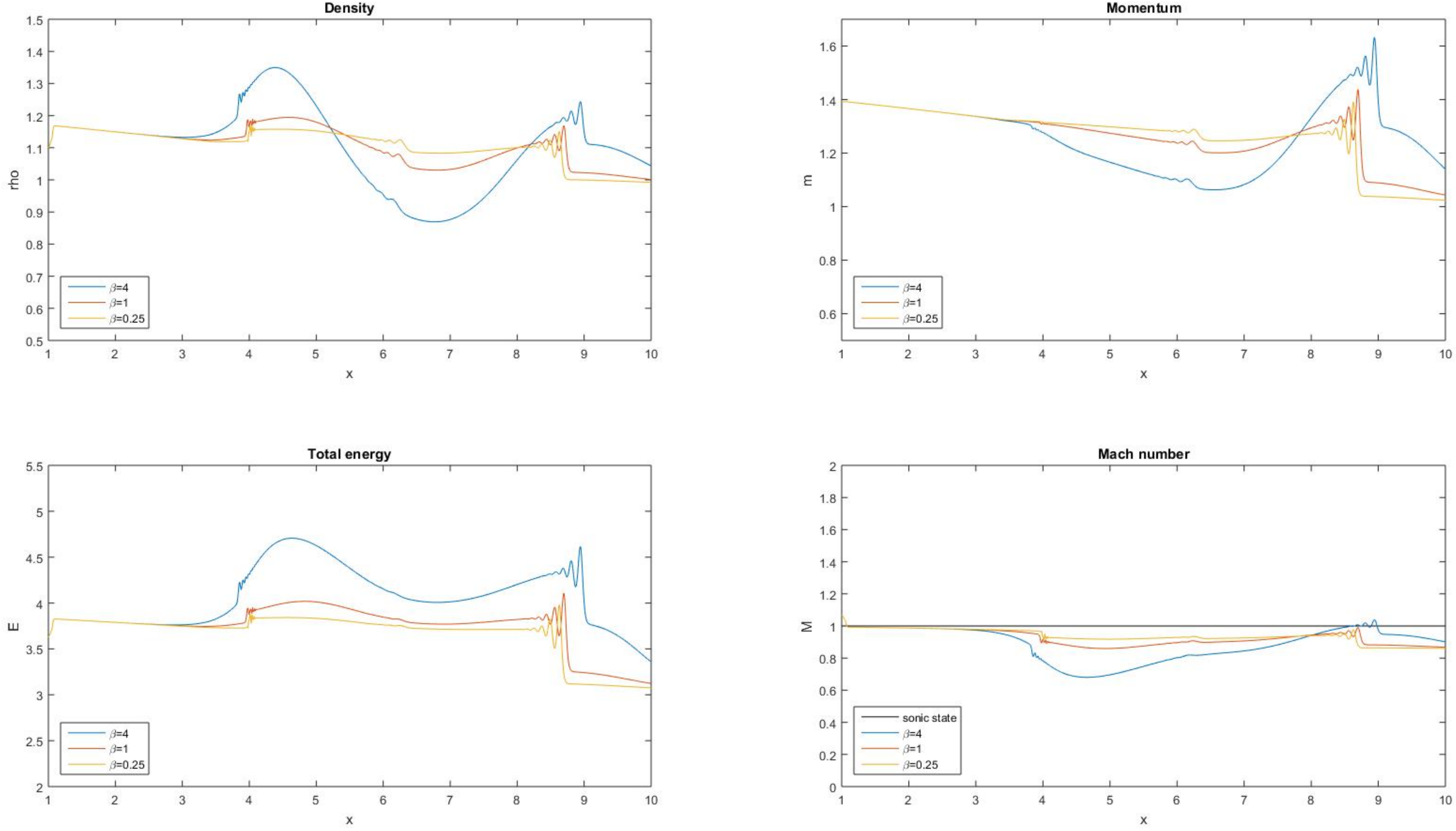}\\
Figure 9. The solutions for \eqref{1.2.2} in subsonic-supersonic case with different $\b$.\medskip
\end{fig}
}
\end{exam}

\begin{exam}
\label{exam5.3}
{\rm(Contraction-Expansion versus Expansion-Contraction)\\
For the expansion-contraction nozzle, if we set the cross section area $a(x)$ as $-\frac{(x-5)^2}{160}+10$, then the solution profile (in both transonic cases) is similar to the contraction-expansion nozzle except minor difference of the value. In order to identify the difference between these two nozzles, we need to shrink the cross area functions as
$$
a_{\pm}(x)=\pm\frac{(x-5)^2}{160}+1,\quad x\in[1,10],
$$
and $a_{\pm}'(x)\equiv 0$ when $x\in\mathbb{R}\backslash[1,10]$. Denote $h_1^{\pm},\ h_2^{\pm}$ are the corresponding functions of $a_{\pm}(x)$ in \eqref{1.2.0}. Using the aforementioned setting, we have
$$
\max_{x\in[1,10]}\Big(\frac{7-\gamma}{3}h_1^+(x)u_0(x)+ \frac{4}{3}h_2^+(x)u_0(x) -\frac{\gamma(\gamma-1)}{\rho_0(x)c^2_0(x)}\beta q(x,0)\Big)<0,\ \text{for }\a>0.066743,\ \b>0,
$$
$$
\max_{x\in[1,10]}\Big(\frac{7-\gamma}{3}h_1^-(x)u_0(x)+ \frac{4}{3}h_2^-(x)u_0(x) -\frac{\gamma(\gamma-1)}{\rho_0(x)c^2_0(x)}\beta q(x,0)\Big)<0,\ \text{for }\a>0.096999,\ \b>0.
$$
In this example, we fix the grid size $\D x=0.01$ the friction parameter $\a=0.1$, the heat parameter $\b=1$. The comparison of contraction-expansion and expansion-contraction nozzles in both transonic cases are shown in Figure \ref{fig10} and Figure \ref{fig11} below. From both figures, we can easily see that the density and momentum distributions are larger on the contraction portion of the nozzle, and lower on the expansion portion of the nozzle. This gives us confidence to our numerical simulation since it fit the common physical knowledge.
\begin{fig}
\label{fig10}
\includegraphics[scale=0.265]{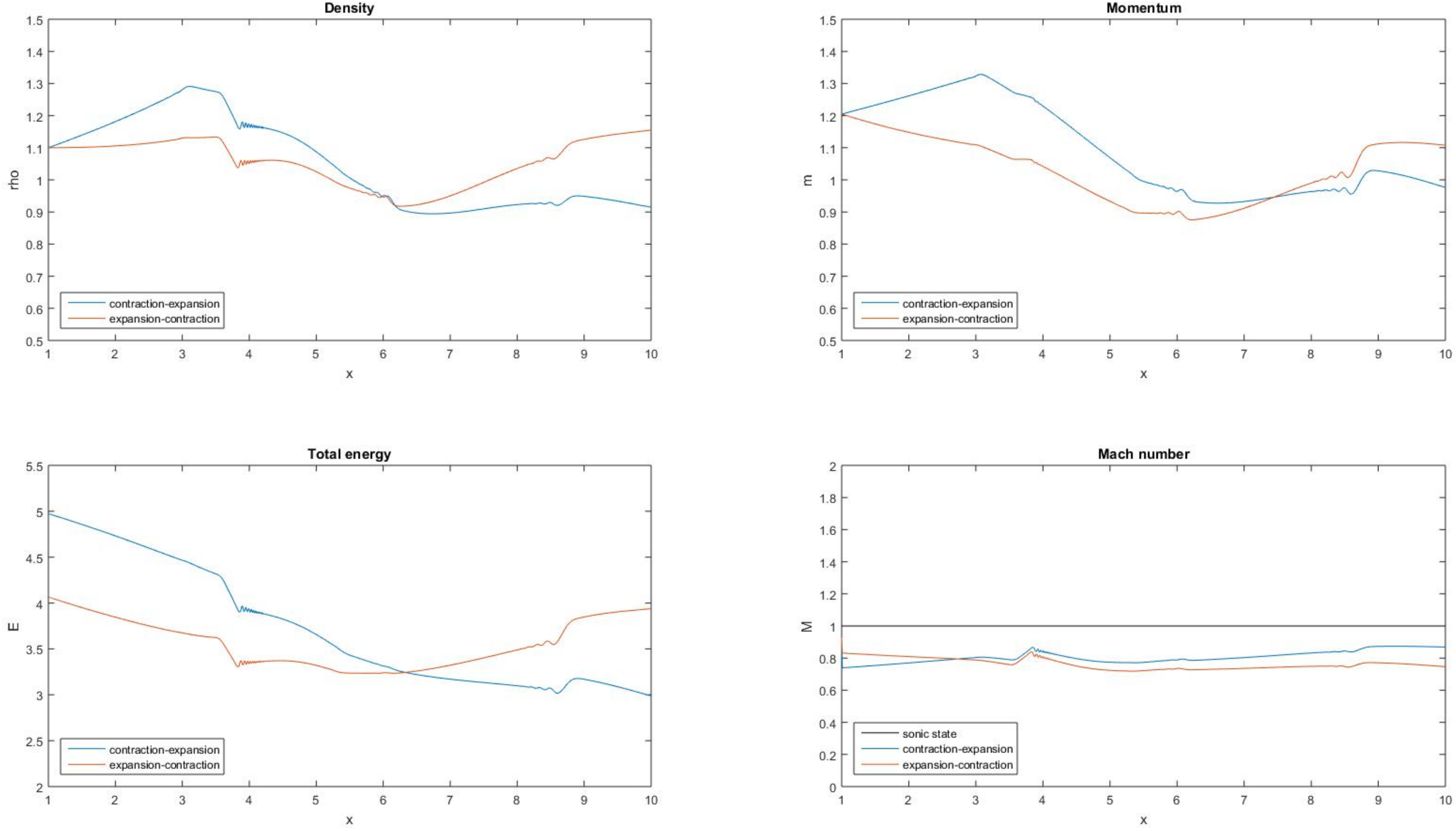}\\
Figure 10. Comparison of the contraction-expansion and expansion-contraction nozzles in subsonic-supersonic case.\medskip
\end{fig}
\begin{fig}
\label{fig11}
\includegraphics[scale=0.265]{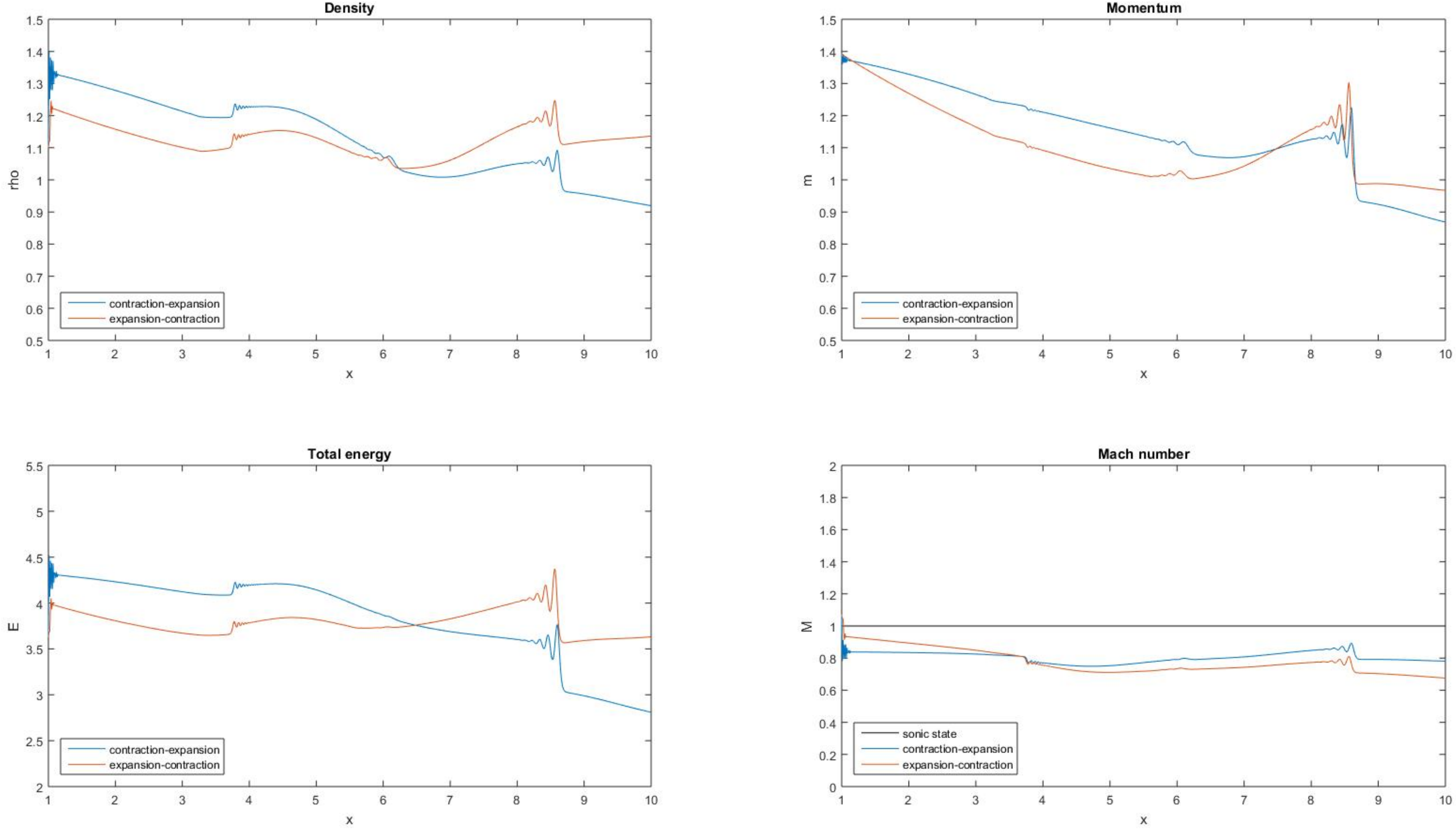}\\
Figure 11. Comparison of the contraction-expansion and expansion-contraction nozzles in supersonic-subsonic case.\medskip
\end{fig}
}
\end{exam}

\vskip 1cm



\end{document}